\documentclass[reqno]{amsart}
\usepackage{fancyhdr}
\usepackage{times}
\usepackage{mathrsfs} 
\usepackage{latexsym,amsopn,amssymb,amsmath}
\usepackage{amsthm}
\usepackage{amsfonts,amsbsy,amscd,stmaryrd}
\usepackage{paralist}

\usepackage{hyperref}
\hypersetup{colorlinks=true,linkcolor=Emerald,citecolor=Maroon,urlcolor=blue} 
\usepackage[usenames,dvipsnames,svgnames,table]{xcolor}
\newcommand{\C}{\mathbb{C}}

\newcommand{\N}{\mathbb{N}}

\newcommand{\R}{\mathbb{R}}

\newcommand{\ca}{\mathcal{A}}

\newcommand{\ce}{\mathcal{E}}

\newcommand{\ch}{\mathcal{H}}

\newcommand{\cu}{\mathcal{U}}

\DeclareMathOperator{\Tr}{Tr}

\def\rmb{\mathrm{b}}
\def\rmc{\mathrm{c}}

\def\dd{\mathrm{d}}

\def\d{\mathrm{d}}

\def\rme{\mathrm{e}}
\def\e{\mathrm{e}}

\def\rmi{\mathrm{i}}
\def\i{\mathrm{i}}

\def\veps{\varepsilon}
\def\vphi{\varphi}

\renewcommand{\proof}{\noindent{\bf Proof. }}

\def\bra#1{\langle{#1}|}
\def\ket#1{|{#1}\rangle}
\def\braket#1#2{\langle{#1}|{#2}\rangle}

\def\jap#1{\langle {#1} \rangle}

\def\what{\widehat}
\def\what#1{\widehat{ #1\,}}

\def\slim{\mbox{\rm s-}\!\lim}

\def\nin{\notin}
\newcommand{\supp}{\mathop{\mathrm{supp}}}

\def\nin{\notin}

\def\qed{\hfill $\Box$\medskip}
\long\def\symbolfootnote[#1]#2{\begingroup%
\def\thefootnote{\fnsymbol{footnote}}\footnote[#1]{#2}\endgroup}
\newtheorem{theorem}{Theorem}[section]
\newtheorem{lemma}[theorem]{Lemma}
\newtheorem{proposition}[theorem]{Proposition}
\newtheorem{corollary}[theorem]{Corollary}

\newtheorem{remark}[theorem]{Remark}

\newtheorem{example}[theorem]{Example}

\newtheorem*{theorem*}{Theorem}
\makeatletter
\newcommand{\labitem}[2]{%
\def\@itemlabel{\textit{#1}.}
\item
\def\@currentlabel{#1}\label{#2}}
\makeatother

\makeatletter
\newcommand{\labitemi}[2]{%
\def\@itemlabel{\textit{#1}}
\item
\def\@currentlabel{#1}\label{#2}}
\makeatother

\begin{document}

\title[Abstract theory of pointwise decay]{Abstract theory of
  pointwise decay with applications to Wave and Schr\"odinger
  equations
}
\author[V. Georgescu]{Vladimir Georgescu$^1$}
\address{(1) V. Georgescu,
  D\'epartement de Math\'ematiques, Universit\'e de Cergy-Pontoise,
  95000 Cergy-Pontoise, France}
\email{vladimir.georgescu@math.cnrs.fr}
\author[M. Larenas]{Manuel Larenas$^2$}
\thanks{The second author was partially supported by NSF DMS-1201394. We thank the referees for valuable comments which helped to improve the manuscript.}
\author[A. Soffer]{Avy Soffer$^2$}
\address{\noindent(2) Rutgers University, Department of Mathematics, 110 Freylinghuysen Road, Piscataway, NJ 08854, U.S.A.}
\email{mlarenas@math.rutgers.edu, soffer@math.rutgers.edu}
\date{}
\begin{abstract}
 We prove pointwise in time decay estimates  via an abstract conjugate
operator method. This is then applied to a large class of  dispersive equations.
\end{abstract}

\maketitle
\setcounter{secnumdepth}{2}
\setcounter{tocdepth}{2} 
\tableofcontents{}

\section{Introduction}

In the study of dispersive equations, linear or nonlinear, one is
faced with the need to quantitatively estimate the decay rate of the
solution in various norms.  The known estimates which play a central
role in the theory of dispersive equations include local decay
estimates, pointwise decay estimates in time, $L^p$ decay estimates
and Strichartz estimates.  More intricate are microlocal estimates and
propagation estimates.  The pointwise decay estimates for
Schr\"odinger operators were proven first in three dimensions and
were obtained for short range potentials \cite{JK}:
$$
\|\jap{x}^{-\sigma} e^{-\i tH}P_c(H)\jap{x}^{-\sigma}\|=\mathcal{O}(t^{-3/2}),
$$
where $\jap{x}^2\equiv 1+|x|^2$, $\sigma$ is large enough, and $P_c(H)$
stands for the projection on the continuous spectral part of $H$.
Here $H\equiv -\Delta+V(x).$

This was later extended by various authors, unified to arbitrary
dimension and allowing resonances at thresholds in \cite{JN1}.  These
estimates play an important role in proving the  $L^p$
decay estimates:
$$
\| e^{-\i tH}P_c(H)\psi\|_{L^{\infty}(\mathbb{R}^n)}\le
ct^{-n/2}\|\psi\|_{L^1(\mathbb{R}^n)}.
$$
Such estimates were proven in some generality in \cite{JSS} in three or
more dimensions. The Kato-Jensen estimate above was used to control
the low energy part of the solution. Moreover, it was remarked by
Ginibre [unpublished], that the Kato-Jensen estimate, when combined
with iterated Duhamel formula, can imply directly a slightly weaker
$L^p$ estimate:
 $$
 \| e^{-\i tH}P_c(H)\psi\|_{L^{\infty}(\mathbb{R}^n)+L^2}\le
 ct^{-n/2}\|\psi\|_{L^1(\mathbb{R}^n)\cap L^2}.
 $$
 This was extended to N-body charge transfer Hamiltonians in
 \cite{RSS1,RSS2}.

 Subsequent works have extended the $L^p$ estimates to all dimensions,
 and general classes of potential perturbations. See
 e.g.\cite{Ya,Sch,RSch,EGG,DSS,KK} and many more.  Common to all these
 results is the explicit use of the kernel of the (unperturbed)
 Hamiltonian. Therefore, such methods are difficult to implement on
 manifolds. In fact, on manifolds most results are of the local decay
 type and Strichartz estimates \cite{RoT,BSo1,BSo2,DR,Ta}.  The
 pointwise decay estimates and the $L^p$ estimates are not known or
 not optimal. In contrast, the abstract method we develop here and in
 a subsequent paper, is not using resolvent estimates. Thus, it is
 applicable in cases where explicit or perturbative methods of
 constructing the resolvent are not suitable. See e.g. Examples
 7.4 - 7.6.

A completely independent method of getting pointwise estimates is based on positive commutator techniques. Mourre's abstract theory to prove decay estimates is based on the Mourre
estimate:
$$
E_I(H)\i [H,A]E_I(H)\ge \theta E_I(H)
$$
for some $\theta>0$, where $E_I(H)$ is the spectral projection on an
interval $I$.  Under regularity assumptions on the pair $H,A$ similar
to ours, Mourre proved that the following local decay estimate holds:
$$
\int \|\jap{A}^{-\sigma}e^{-\i tH}E_I(H)\psi\|^2 dt \le c\|\psi\|^2.
$$
Mourre's method was then generalized in \cite{JMP} and later in
\cite{BouG,BGS} to prove pointwise in time decay estimates. See the
discussion and details at the end of Section \ref{s:calpha}.  Later,
in \cite{SS,HSS,Ger} a new, time dependent method was developed to
prove pointwise decay estimates, local decay and other propagation
estimates, starting only from the Mourre estimate.  The propagation
estimates of \cite{SS}, which also allowed some classes of time
dependent Hamiltonians, and optimized in \cite{HSS} read:
 $$
 \|F(A\le \epsilon t)e^{-\i tH}E_I(H)\psi\|=\mathcal{O}(t^{-m})\|\jap{A}^{m+1}\psi\|,
 $$
 $$
 \|F(A\ge  bt)e^{-\i tH}E_I(H)\psi\|=\mathcal{O}(t^{-m})\|\jap{A}^{m+1}\psi\|,
 $$
 for all $m$ depending only on the regularity of the potential and the
 localization of the initial data.  $\epsilon, b$ depend only on the
 interval $I$, and the constant $\theta$.  This method and results
 provide a powerful tool to spectral and scattering theory, including
 the N-body systems and Quantum Field theory.  However, the positivity
 condition in the Mourre estimate breaks down at (finite) thresholds.

Another way around this problem are the Morawetz type estimates. They apply
at thresholds, but limited to nontrapping potentials. The extension to
repulsive potentials and low dimension was established as well in some
cases.
Mourre's method was extended in many works to include thresholds
\cite{BouG,MR,MT,MRT1,MRT2,BG,GJ1,GJ2,RT1,RT2,RT3,RT4,Ti,Sa,BH,So}.
They are based on requiring the Mourre estimate to hold with a lower bound given by some positive operator, which is not
a multiple of the identity.
However these methods so far could not be versatile enough to include many
common systems, mainly due to complicated assumptions or the use of
abstract weighted spaces. They only imply local decay estimates of the type mentioned above.

In this work and forthcoming papers, we develop a new abstract theory
to prove {\it pointwise} decay estimates in weighted spaces, starting
only from a general commutator identity that should be satisfied by
the Hamiltonian.  We will show that for Schr\"odinger type equations
generated by an abstract Hamiltonian $H$, as well as Klein Gordon and
wave equations, pointwise decay estimates of the Kato-Jensen type hold
using the following assumptions:

\begin{itemize}

\item[a)] The pair $H,A$ with $A$ self adjoint, should satisfy regularity
conditions similar to Mourre's method.

\item[b)] A commutator identity of the type $\i[H,A]=\theta(H)+Q$
with $[Q,H]=0$ and $Q$ is $H$-\text{bounded}.

\item[c)] The subspace $\mathcal{E}$ of vectors which satisfy local decay (as
above) and are in the domain of $A$, is known in some explicit sense,
e.g. that it is all vectors in the domain of $\langle A\rangle^m$ and in
some invariant subspace under the dynamics (generated by $H$).

\end{itemize}

These conditions differ in several aspects from the standard theories
above.  The commutator condition is restrictive, e.g. it is unstable
under small perturbations of the Hamiltonian.  However, we show in a
subsequent work that local decay estimates are sufficient to absorb
the effects of classes of perturbations; it is done by constructing a
modified conjugate operator $\tilde A$, that satisfies the above
conditions.  So, in fact, the main condition is to identify, in some
explicit way the subspace of vectors which satisfies local decay.  It
should be noted that positivity of commutators is not used. But,
positivity would imply local decay, and better decay estimates in
certain cases.  To prove local decay estimates, either some resolvent
bounds or some weakly positive commutators can be employed. In
practice, this can be achieved by using Morawetz type estimates or
elementary perturbative resolvent estimates, relying only on the Fredholm
alternative and compactness arguments.

We will then give several examples to show that such estimates follow
effortlessly from the general theory.
In followup papers we extend this method to include perturbations of $H$
of the types described above.

\section[short]{Evanescent states}\label{s:eva}

Let $H$ be a self-adjoint operator on the Hilbert space $\ch$ with
spectral measure $E$. If $u\in\ch$ let $E_u$ be the measure
$E_u(J)=\|E(J)u\|^2$ and $\psi_u:\R\to\C$ the function
$\psi_u(t)=\braket{u}{\e^{\i tH}u}=\int_\R\e^{\i
  t\lambda}E_u(\d\lambda)$.  We are interested in vectors $u$ such
that $\psi_u(t)\to 0$ as $t\to\infty$ and in the rapidity of this
decay.

Note that $|\psi_u(t)|^2$ is a physically meaningful quantity if we
think of $H$ as the Hamiltonian of a system whose state space is
$\ch$. Indeed, if $u,v$ are vectors of norm one then
$|\braket{v}{\e^{\i tH}u}|^2$ is the probability of finding the system
in the state $v$ at moment $t$ if the initial state is $u$, hence
$|\psi_u(t)|^2$ \emph{is the probability that at moment $t$ the system
  be in the same state $u$ as at moment $t=0$}.

\begin{remark}\label{re:comments}{\rm In this paper we are interested
    in the decay properties of the functions $\psi_u$ for $u$ in the
    absolute continuity subspace $\ch_H^{\mathrm{ac}}$ of $\ch$
    relatively to $H$. We shall see that $\psi_u\in L^2(\R)$ for $u$
    in a dense subspace of $\ch_H^{\mathrm{ac}}$ but in rather simple
    cases it may happen that $\psi_u\in L^1(\R)$ only for
    $u=0$. Formally speaking, the physically interesting quantity
    $|\psi_u(t)|^2$ generically decays more rapidly than
    $\jap{t}^{-1}$ but not as rapidly as $\jap{t}^{-2}$.  Our results
    concern mainly the rate of this decay, for example we give
    conditions such that $|\psi_u(t)|^2$ is really dominated by
    $\jap{t}^{-1}$, not only in an $L^2$ sense.
}\end{remark}

Since $\psi_u$ is (modulo a constant factor) the Fourier transform of
$E_u$, there is a strong relation between the decay of $\psi_u$ and
the smoothness of $E_u$.  If $u$ is absolutely continuous with respect
to $H$ then $\psi_u\in C_0(\R)$ (space of continuous functions which
tend to zero at infinity). However, the decay may be quite slow if
$E_u$ is not regular enough.

\begin{example}\label{ex:slow}
  Let $\Lambda$ be a real compact set with empty interior and strictly
  positive Le\-bes\-gue measure and let $H$ be the operator of
  multiplication by $x$ in $\ch=L^2(\Lambda,\d x)$. Then the spectrum
  of $H$ is purely absolutely continuous but $\psi_u\nin L^1(\R)$ for
  all $u\in\ch\setminus\{0\}$. {\rm Indeed, if $u\in\ch$ and we extend
    it by zero outside $\Lambda$ then $\psi_u(t)=\int \e^{\i
      tx}|u(x)|^2 \d x$ hence if $\psi_u$ is integrable then $|u|^2$
    is the Fourier transform of an integrable function, so it is
    continuous, so the set where $|u(x)|^2\neq0$ is open and contained
    in $\Lambda$, hence it is empty.  }\end{example}

On the other hand, if $H$ has an absolutely continuous component then
there are plenty of $u$ such that $\psi_u\in L^2(\R)$: indeed, observe
that \emph{$\psi_u\in L^2(\R)$ if and only if $E_u$ is an absolutely
  continuous measure with derivative $E_u'\in L^2(\R)$ and then
  $\|\psi_u\|_{L^2}=\sqrt{2\pi}\|E_u'\|_{L^2}$}.

More generally, if we denote $E_{v,u}$ the complex measure
$E_{v,u}(J)=\braket{v}{E(J)u}$ then $\braket{v}{\e^{\i
    tH}u}=\int\e^{\i t\lambda}E_{v,u}(\dd\lambda)$ hence the left hand
side belongs to $L^2(\R)$ if and only if the measure $E_{v,u}$ is
absolutely continuous and has square integrable derivative $E_{v,u}'$
and then we have $\int_\R |\braket{v}{\e^{\i tH}u}|^2 \d t=2\pi\int
|E_{v,u}'(\lambda)|^2\dd\lambda$. It is easy to prove the inequality
$|E_{v,u}'(\lambda)|^2\leq E_{v}'(\lambda)E_{u}'(\lambda)$, see
(\cite[page 1002]{Ro} or \cite[Section 3.5]{BW}) and this implies
$E_{u+v}^{\prime 1/2}\leq E_u^{\prime 1/2}+E_v^{\prime 1/2}$. Thus, if
we set for any $u\in\ch$
\[ [u]_H= \left(\int_\R |\psi_u(t)|^2 \d t \right)^{1/4} =
\left(2\pi\int_\R E_u'(\lambda)^2 \d\lambda \right)^{1/4}
\]
then
\begin{equation}\label{eq:ce}
\ce\equiv\ce(H)=\{u\in\ch \mid [u]_H<\infty\}
\end{equation}
\emph{is a dense linear subspace of the absolutely continuous subspace
  of $H$ and $[\cdot]_H$ is a complete norm on it}.  We mention that
the relation
$|E_{v,u}'(\lambda)|^2\leq E_{v}'(\lambda)E_{u}'(\lambda)$ also
implies
\begin{equation}\label{eq:ineq}
\left(\int_\R |\braket{v}{\e^{\i tH}u}|^2 \d t \right)^{1/2} \leq [v]_H [u]_H .
\end{equation}

\begin{lemma}\label{lm:scom}
  If $J\in B(\ch)$ commutes with $H$ then $J\ce\subset\ce$ and
  $[Ju]_H\leq \|J\| [u]_H$. If $J_n=\theta_n(H)$ with $\{\theta_n\}$ a
  uniformly bounded sequence of Borel functions such that $\lim_{n}
  \theta_n(\lambda)=1$ for all $\lambda\in\R,$ then for any
  $u\in\ce$ we have $\lim_n[J_nu]_H=[u]_H$.
\end{lemma}

\proof For the first part we use $E'_{Ju}(\lambda)\leq
\|J\|^2E'_u(\lambda)$ (which is obvious) while for the second part
$E'_{\theta_n(H)u}(\lambda)=\theta_n^2(\lambda)E'(\lambda)$ and the
dominated convergence theorem.  \qed

The quantity $\int_\R |\psi_u(t)|^2 \d t$ has a simple physical
interpretation in the quantum setting: \emph{if $u,v$ are two state
  vectors then $\int_\R |\braket{v}{\e^{\i tH}u}|^2 \d t$ is the total
  time spent by the system in the state $v$ if the initial state is
  $u$}. Hence we may say that
$\int_\R |\braket{u}{\e^{\i tH}u}|^2 \d t$ is the \emph{lifetime of
  the state $u$}.  The elements of $\ce(H)$ are those of finite
lifetime, or states in which the system spends a finite total time. We
might call them \emph{self evanescent states}, and they are absolutely
continuous with respect to $H$. Note that there is a Schr\"odinger
Hamiltonian $H$ and there is a state $u$ in the singularly continuous
subspace of $H$ such that $\psi_u(t)=\mathcal{O}(|t|^{-1/2+\varepsilon})$ for
any $\varepsilon>0$ \cite{Si}.

Another interesting class $\ce_\infty\equiv\ce_\infty(H)$ is that of
\emph{evanescent states} defined by the condition
$\int_\R |\braket{v}{\e^{\i tH}u}|^2 \d t<\infty$ for any $v$:
\emph{such a state $u$ spends a finite time in any state $v$}. The
evanescent states disappear (or go to infinity) in a natural quantum
mechanical sense, which explains the fundamental role they play in the
Rosenblum Lemma \cite{Ro} and later on in the Birman-Kato trace class
scattering theory.  A simple argument shows that $\ce_\infty$ is the
linear subspace of $\ce$ consisting of vectors $u$ such that $E_u'$ is
a bounded function.  In particular, $\ce_\infty$ is dense in the
absolutely continuity subspace of $H$.

\begin{example}\label{ex:eva}{\rm If $H=q=$ operator of multiplication
    by $x$ in $L^2(\R,\d x)$ then $\ce(q)=L^4(\R)$ and
    $\ce_\infty(q)=L^\infty(\R)$. Indeed, $\braket{u}{\e^{\i
        tq}u}=\int_\R \e^{\i tx}|u(x)|^2\d x$ is an $L^2$ function of
    $t$ if and only if $|u|^2\in L^2$ and then
    $[u]_q=(2\pi)^{1/4}\|u\|_{L^4}$.  On the other hand,
    $\braket{v}{\e^{\i tq}u}=\int_\R \e^{\i tx}\bar{v}(x)u(x)\d x$ is
    an $L^2$ function of $t$ for any $v\in L^2$ if and only if
    $\bar{v}u\in L^2$ for any $v\in L^2$ hence if and only if $u\in
    L^\infty$.  }\end{example}

\section{Notes on commutators}\label{s:c1a}

Let $A$ be a self-adjoint operator on a Hilbert space $\ch$.  If $S$
is a bounded operator on $\ch$ then we denote $[A,S]_\circ$ the
sesquilinear form on $D(A)$ defined by
$[A,S]_\circ(u,v)=\braket{Au}{Sv}-\braket{u}{SAv}$. As usual, we set
$[S,A]_\circ=-[A,S]_\circ$, $[S,\i A]_\circ=\i [S,A]_\circ$, etc.  We
say that \emph{$S$ is of class $C^1(A)$}, and we write $S\in C^1(A)$,
if $[A,S]_\circ$ is continuous for the topology induced by $\ch$ on
$D(A)$ and then we denote $[A,S]$ the unique bounded operator on $\ch$
such that $\braket{u}{[A,S]v}=\braket{Au}{Sv}-\braket{u}{SAv}$ for all
$u,v\in D(A)$.  It is easy to show that $S\in C^1(A)$ if and only if
$SD(A)\subset D(A)$ and the operator $SA-AS$ with domain $D(A)$
extends to a bounded operator $[A,S]\in B(\ch)$. Moreover, $S$ is of
class $C^1(A)$ if and only if the following equivalent conditions are
satisfied
\begin{compactenum}
\item the function $t\mapsto \e^{-\i t A}S\e^{\i t A}$ is Lipschitz in
  the norm  topology in $ B(\ch)$
\item the function $t\mapsto \e^{-\i t A}S\e^{\i t A}$ is of class
  $C^{1}$ in the strong operator topology
\end{compactenum}
and then we have $[S,\i A]=\frac{ d}{ dt}\e^{-\i t A}S\e^{\i t
  A}|_{t=0}$.

Clearly $C^1(A)$ is a $*$-subalgebra of $ B(\ch)$ and the usual
commutator rules hold true: for any $S, T\in C^{1}(A)$ we have
$[A,S]^*=-[A,S^*]$ and $[A,ST]= [A,S]T+ S[A,T]$, and if $S$ is
bijective then $S^{-1}\in C^1(A)$ and $ [A,S^{-1}]= -
S^{-1}[A,S]S^{-1}$.

We often abbreviate $S'=[S,\i A]$ if the operator $A$ is
  obvious from the context. Then we may write $(S')^*=(S^*)'$, $(S
  T)'= S'T+ ST'$, and $(S^{-1})'= - S^{-1}S'S^{-1}$.

We consider now the rather subtle case of unbounded operators. Note
that we always equip the domain of an operator with its graph
topology.  If $H$ is a self-adjoint operator on $\ch$ then
$[A,H]_\circ$ is the sesquilinear form on $D(A)\cap D(H)$ defined by
$[A,H]_\circ(u,v)=\braket{Au}{Hv}-\braket{Hu}{Av}$. By analogy with
the bounded operator case, one would expect that requiring denseness
of $D(A)\cap D(H)$ in $D(H)$ and continuity of $[A,H]_\circ$ for the
graph topology of $D(H)$ would give a good $C^1(A)$ notion. For
example, this should imply the validity of the virial theorem, nice
functions of $H$ (at least the resolvent) should also be of class
$C^1$, etc. However this is not true, as the following example from
\cite{GG} shows.

\begin{example}\label{ex:exo}{\rm In $\ch=L^2(\R,\dd x)$ let $q=$
    operator of multiplication by $x$ and $p=-\i\frac{d}{dx}$.  Let
    $A=\e^{\omega p}-p$ and $H=\e^{\omega q}$ with
    $\omega=\sqrt{2\pi}$. This value of $\omega$ is chosen because
    $\e^{\omega p}\e^{\omega q}=\e^{\omega q}\e^{\omega p}$ on a very
    large set although the operators $\e^{\omega p}$ and $\e^{\omega
      q}$ do not commute. Then $D(A)\cap D(H)$ is dense in both $D(A)$
    and $D(H)$ (moreover, $D(H)\cap D(HA)$ is dense in $D(H)$), one
    has $[H,\i A]_\circ=\omega H$ on $D(A) \cap D(H)$, but
    $(H+\i)^{-1}\nin C^1(A)$. }\end{example}

A convenient definition of the $C^1(A)$ class for any self-adjoint
operator is as follows. Let $ R(z)= (H-z)^{-1}$ for $z$ in the
resolvent set $\rho(H)$ of $H$.  We say that \emph{$H$ is of class
  $C^1(A)$} if $R(z)\in C^1(A)$ for some (hence for all)
$z\in\rho(H)$.  In this case the space $R(z)D(A)$ is independent of
$z\in \rho(H)$, it is a core of $H$, and is a dense subspace of
$D(A)\cap D(H)$ for the intersection topology, i.e.\ for the norm
$\|u\|+\|Au\|+\|Hu\|$. Moreover:

\begin{proposition}\label{pr:abg}
Let $A,H$ be self-adjoint operators on a Hilbert space $\ch$.
\begin{compactenum}
\item $H$ is of class $C^{1}(A)$ if and only if the next two
  conditions are satisfied:
\begin{compactenum}
\item $[A,H]_\circ$ is
  continuous for the topology induced by $D(H)$ on $D(A)\cap D(H)$,
\item there is $z\in \rho(H)$ such that $\{u\in D(A) \mid
  R(z)u\in D(A)\}$ is a core for $A$.
\end{compactenum}
\smallskip
\item If $H\in C^1(A)$ then $D(A)\cap D(H)$ is dense in $D(H)$ hence
  $[A,H]_\circ$ has a unique extension to a continuous sesquilinear
  form $[A,H]$ on $D(H)$. We have:
  \begin{equation}\label{eq:ahcom} [A,R(z)] = - R(z) [A,H] R(z) \quad
    \forall\, z\in\rho(H).
\end{equation}
\end{compactenum}
\end{proposition}

This is Theorem 6.2.10 in \cite{ABG}. The condition (a) above is quite
easy to check in general but not condition (b) because it involves a
certain knowledge of the resolvent of $H$, which is a complicated
object. We now describe criteria which allow one to avoid this
problem.

We denote $\ch^1=D(H)$ (equipped with the graph topology) and
$\ch^{-1}=D(H)^*$ its adjoint space.  The identification of the
adjoint space $\ch^*$ of $\ch$ with itself via the Riesz Lemma gives
us a scale $\ch^1\subset\ch\subset\ch^{-1}$ with continuous and dense
embeddings. If we define $\ch^{s}:=[\ch^1,\ch^{-1}]_{(1-s)/{2}}$ for
$-1\leq s\leq1$ by complex interpolation then $(\ch^{s})^*=\ch^{-s}$
for any $s$ and $\ch^{1/2}$ is just $D(|H|^{1/2})$. Finally, we have
continuous and dense embeddings
\[
\ch^1\subset \ch^{1/2}\subset\ch\subset \ch^{-1/2}\subset\ch^{-1}  .
\]
If $H\in C^1(A)$ the continuous sesquilinear
form $[A,H]$ on $D(H)$ is then identified with a linear continuous
operator $\ch^1\to\ch^{-1}$ and this is useful for example
because it gives a simple interpretation to supplementary conditions
like $[A,H]\ch^1\subset\ch$. Observe that
$$
H':=[H,\i A]:\ch^1\to\ch^{-1}
$$
is a continuous symmetric operator. Now the following assertions are
consequences of \cite[Theorem 6.3.4, Lemma 7.5.3]{ABG} and \cite[Lemma
2]{GG}.

\begin{enumerate}

\item If $\e^{\i tA}\ch^1\subset \ch^1\,(\forall t)$ then $H\in C^1(A)$ if
  and only if condition (a) in part (1) of Proposition \ref{pr:abg}  is satisfied.

\item \label{p:2} If $H\in C^1(A)$ and $H'\ch^1\subset\ch$ then
  $\e^{\i tA}\ch^1\subset \ch^1 \,(\forall t)$ and the restrictions
  $\e^{\i tA}|\ch^1$ give a strongly continuous group of operators on
  the Hilbert space $\ch^1$.

\item If $\e^{\i tA}\ch^1\subset \ch^1\,(\forall t)$ then
  $D(A,\ch^1)=\{u\in\ch^1\cap D(A)\mid Au\in\ch^1\}$ is a dense
  subspace of $\ch^1$ and $H$ is of class $C^1(A)$ with
  $H'\ch^1\subset \ch$ if and only if
  $|\braket{Au}{Hv}-\braket{Hu}{Av}|\leq C \|u\|_{\ch}\|v\|_{\ch^1}$
  for all $u,v\in D(A,\ch^1)$.

\item Assume $\e^{\i tA}\ch^{1/2}\subset \ch^{1/2}\,(\forall t)$. Then
  $D(A,\ch^{1/2}):=\{u\in\ch^{1/2}\cap D(A)\mid Au\in \ch^{1/2}\}$ is
  dense in $\ch^{1/2}$ and if the quadratic form
  $\braket{Hu}{Au}-\braket{Au}{Hu}$ on $D(A,\ch^{1/2})$ is continuous
  for the topology induced by $\ch^{1/2}$ then $H\in C^1(A)$.

\end{enumerate}

We mention that Hypotheses $1$, $2'$ and $3$ on page 62 of \cite{CFKS}
imply that $H$ is of class $C^1(A)$, cf.\ relation (4.10) there.

We now give some ``pathological'' examples which clarify the notion of
$C^1$ regularity.

\begin{example}\label{ex:exp}{\rm Let $\ch=L^2(\R)$ and $A=p$. It is
    clear that the operator of multiplication by a rational real
    function is of class $C^1(p)$, in fact of class $C^\infty(p)$ in a
    natural sense. For example, if $H=q^{-m}$ then
    $(H+\i)^{-1}=q^m(1+\i q^m)^{-1}$ is clearly a bounded operator of
    class $C^1(p)$ if $m\in\N$ and $[q^{-m},\i p]=mq^{-m-1}$ as
    continuous forms on $D(q^{-m})$. The worst case is attained when
    $m=1$: then $H'=H^2$ hence $H'+\i:\ch^1\to\ch^{-1}$ is an
    isomorphism, in particular $H'\ch^1$ is not included in any of the
    smaller spaces $\ch^{s}$ with $s>-1$.  If $m\geq1$ is an odd
    integer then $H$ is of class $C^1(A)$ and $H'=mH^{1+1/m}$ where
    $x^{1/m}:=-|x|^{1/m}$ if $x<0$; now we have
    $H'\ch^1\subset\ch^{-1/m}$ and this is optimal.  }\end{example}

\begin{remark}\label{re:note}{\rm Example \ref{ex:exp} shows that
    \emph{if $H\in C^1(A)$ then neither $\e^{\i tA}$ nor $(A+\i
      \lambda)^{-1}$ leave invariant $D(H)$ in general}. }\end{remark}

If $H\in C^1(A)$ then $D(A)\cap D(H)$ is dense in $D(H)$ but \emph{is
  not dense in $D(A)$} in general.

\begin{example}\label{ex:dense}{\rm Let $H=q^{-m}$ with $m\geq1$ and
    $A=p$ as in Example \ref{ex:exp}. Then $D(A)$ is the Sobolev space
    consisting of functions $u\in L^2(\R)$ with derivative $u'\in
    L^2(\R)$, so we have $D(A)\subset C_0(\R)$ continuously.  Thus if
    $u\in D(A)\cap D(H)$ then $u$ is a continuous function such that
    $\int|u(x)|^2x^{-2m}\d x<\infty$ which implies $u(0)=0$. But
    $\{u\in D(A)\mid u(0)=0\}$ is a closed hyperplane of codimension
    one in the Hilbert space $D(A)$.  }\end{example}

Given $\varepsilon >0$, taking $m$ large in the preceding example we see that for any
($\varepsilon>0$) there is a self-adjoint operator $H$ of class $C^1(A)$
with $H'\ch^1\subset\ch^{-\varepsilon}$ such that $D(A)\cap D(H)$ is
not dense in $D(A)$. Thus \emph{the next result is optimal}.

\begin{proposition}\label{pr:dense}
  If $H\in C^1(A)$ and $H'\ch^1\subset\ch$ then $D(A)\cap D(H)$ is
  dense in $D(A)$. More precisely, if we set
  $R_\varepsilon=(1+\i\varepsilon H)^{-1}$ for $\varepsilon>0$ then
  $R_\varepsilon D(A)\subset D(A)\cap D(H)$ and $\slim R_\varepsilon
  =1$ in the Hilbert space $D(A)$.
\end{proposition}

\proof We have $R_\varepsilon D(A)\subset D(A)\cap D(H)$ and
$[A,R_\varepsilon] = \varepsilon R_\varepsilon H'R_\varepsilon$ by
Proposition \ref{pr:abg}. Then \( \varepsilon\|H'R_\varepsilon\|\leq
\|H'R_1\|\|(\varepsilon +\i \varepsilon H) R_\varepsilon\| \leq C \)
and $\varepsilon H'R_\varepsilon u=\varepsilon H'R_1 R_\varepsilon (1+\i
H)u\to0$ if $u\in D(H)$. Thus $\slim_{\varepsilon\to0}
[A,R_\varepsilon] =0$ hence $AR_\varepsilon u\to Au$ for any $u\in
D(A)$.  \qed

This $C^1(A)$ property transfers from $H$ to some functions of $H$:
for example, it is easy to prove that $\vphi(H)\in C^1(A)$ if
$\vphi\in C^2(\R)$ and
$|\varphi(\lambda)|+|\vphi'(\lambda)|+|\vphi''(\lambda|\leq
C\jap{\lambda}^{-2}$. But obviously $\e^{\i H}\nin C^1(A)$ in general.

\begin{theorem}\label{th:exp}
  Let $H$ be a self-adjoint operator of class $C^1(A)$ and
  $t\in\R$. Then the restriction of $[A,\e^{\i tH}]_\circ$ to
  $D(A)\cap D(H)$ extends to a continuous form $[A,\e^{\i tH}]$ on
  $D(H)$ and, in the strong topology of the space of sesquilinear
  forms on $D(H)$, we have:
\begin{equation}\label{eq:exp}
[\e^{\i tH},A]=\int_0^t \e^{\i(t-s)H}H' \e^{\i sH} \dd s .
\end{equation}
\end{theorem}
\proof Clearly it suffices to assume $t=1$. For $n\geq1$ integer let
$R_n=(1-\i H/n)^{-1}$. Then $R_n$ has norm $\leq1$ and $\e^{\i
  H}=\slim_{n\to\infty}R_n^n$ in both spaces $\ch$ and $D(H)$. Since
$H$ is of class $C^1(A)$ we have $R_n\in C^1(A)$ and
$[A,R_n]=\frac{\i}{n}R_n[A,H]R_n$, so $R_n^n\in C^1(A)$ and
\[
[A,R_n^n]=\sum_{k=0}^{n-1} R_n^k[A,R_n]R_n^{n-1-k} =
\frac{\i}{n}\sum_{k=1}^{n} R_n^{k} [A,H]  R_n^{n+1-k}  .
\]
It is clear that $\braket{u}{[A,R_n^n]v}\to\braket{u}{[A,\e^{\i
    H}]_0v}$ as $n\to\infty$ for all $u,v\in D(A)$. Thus it remains to
be shown that for all $u,v\in D(H)$:
\begin{equation}\label{eq:lim}
\frac{1}{n}\sum_{k=1}^{n} \braket{R_n^{*k}u}{[A,H] R_n^{n+1-k}v} \to
\int_0^1 \braket{\e^{-\i sH}u}{[A,H]\e^{\i(1-s)H}v} \dd s .
\end{equation}
We have
\[
\Big\|\sum_{k=1}^{n} R_n^{k} [A,H] R_n^{n+1-k}\Big\|_{\ch^1\to\ch^{-1}}
\leq n \|[A,H]\|_{\ch^1\to\ch^{-1}}
\]
hence it suffices to prove that \eqref{eq:lim} holds for $u,v$ in a
dense subspace of $D(H)$. So we may assume that $u,v$ have compact
support with respect to $H$.

Let $a$ be a number such that $|\log(1+z)-z|\leq a|z|^2$ if $z\in\C$
and $|z|<1/2$. If $\phi_n(x)=(1-\i x/n)^{-1}$ then for $x$ in a real
compact set, $1\leq k\leq n$, and $n$ large, we have
\[
|\phi_n(x)^k-\e^{\i\frac{kx}{n}}|
=|\e^{k\log(1-\i\frac{x}{n}) + \i\frac{kx}{n} } -1|
\leq C k |\log(1-\i x/n) +\i x/n|
\leq Ck a |x/n|^2
\]
where $C$ is a number depending only on the set where varies $x$. Thus
the last term above is an $\mathcal{O}(x^2/n)$ and so we get
$\|R_n^{*k}u-\e^{-\i \frac{k}{n}H}u\|_{D(H)}=\mathcal{O}(n^{-1})$. A
similar argument gives
$\|R_n^{n+1-k}v-\e^{\i\frac{n+1-k}{n}H}V\|_{D(H)}=\mathcal{O}(n^{-1})$. Hence:
\[
\frac{1}{n}\sum_{k=1}^{n} \braket{R_n^{*k}u}{[A,H] R_n^{n+1-k}v}
= \frac{1}{n}\sum_{k=1}^{n}
\braket{\e^{-\i\frac{k}{n}H}u}{[A,H] \e^{-\i\frac{k}{n}H}
\e^{\i\frac{n+1}{n} H}v}
+\mathcal{O}(n^{-1}).
\]
Finally, we have $\e^{\i\frac{n+1}{n} H}v\to \e^{\i H}v$ in $D(H)$ and the
$D(H)$-valued functions $s\mapsto\e^{-\i sH}u$ and $s\mapsto\e^{-\i
  sH}v$ are continuous. This proves \eqref{eq:exp}.
\qed

The relation \eqref{eq:exp} also holds in $B(D(H),D(H)^*)$ in the
strong topology and then one may easily prove relations like the next
one hold in $B(\ch)$:
\begin{equation}\label{eq:expr}
[A,\e^{\i tH}R(z)^2]= R(z)[A,\e^{\i tH}]R(z)
+[A,R(z)]\e^{\i tH}R(z) + \e^{\i tH} R(z)[A,R(z)] .
\end{equation}
If $H'D(H)\subset\ch$ then the right hand side of \eqref{eq:exp} will
clearly belong to $B(D(H),\ch)$ hence we shall also have $[A,\e^{\i
  tH}]\in B(D(H),\ch)$ and \eqref{eq:exp} will hold strongly in
$B(D(H),\ch)$.

We say that \emph{$H'$, or $[A,H]$, commutes with $H$} if for any
$t\in\R$ the relation $H'\e^{\i tH}=\e^{\i tH}H'$ holds in
$B(D(H),D(H)^*)$. This is clearly equivalent to
$H'\varphi(H)=\varphi(H) H'$ for any bounded Borel function
$\varphi:\R\to\C$.  Note also that $H'$ commutes with $H$ if and only
if there is $z\in\rho(H)$ such that $[A,R(z)]$ commutes with $R(z)$
(this condition is independent of $z$). If we set $R=R(z)$, we then
have $R'=-RH'R=-H'R^2$.

If $H'$ commutes with $H$ then Theorem \ref{th:exp} can be
significantly improved. If $k\in\N$ let $C^k_\rmb(\R)$ be the space of
functions in $C^k(\R)$ whose derivatives of orders $\leq k$ are
bounded.

\begin{proposition}\label{pr:com}
  Let $H$ be self-adjoint of class $C^1(A)$ such that $H'$ commutes
  with $H$ and let $\varphi\in C^1_\rmb(\R)$. Then the restriction of
  $[A,\varphi(H)]_\circ$ to $D(A)\cap D(H)$ extends to a continuous
  form $[A,\varphi(H)]$ on $D(H)$ and
  $[A,\varphi(H)]=[A,H]\vphi'(H)=\vphi'(H)[A,H]$.  In other terms:
\begin{equation}\label{eq:deriv}
  \varphi(H)'=\varphi'(H)H'=H'\varphi'(H) , \quad
  \text{in particular }  (\e^{\i tH})'=\i t H'\e^{\i tH}   .
\end{equation}
\end{proposition}

\proof Due to Theorem \ref{th:exp} we have $[A,\e^{\i tH}]=\i t
[A,H]\e^{\i tH}$ for any real $t$, hence the proposition is true if
$\varphi(\lambda)=\e^{\i t\lambda}$. This clearly implies the
proposition if $\varphi$ is the Fourier transform of a bounded measure
$\what\varphi$ such that $\int |x\what\varphi(x)|\dd x<\infty$. The
general case follows by a standard limiting procedure. \qed

\begin{example}\label{ex:expo}{\rm Consider once again the situation
    from Example \ref{ex:exp}. Then $[p,\e^{\i H}]_\circ$ is a
    restriction of $-mq^{-m-1}\e^{\i H}$ hence is not a bounded
    operator but it extends to a continuous form on $D(H)$.  In the
    worst case $m=1$ we get $[\e^{\i H},A]=H^2 \e^{\i H}$, hence
    \emph{the result of Theorem \ref{th:exp} is optimal}. If
    $\vphi$ is a $C^1$ function then $\varphi(H)'=H^2\vphi'(H)$ hence
    $\varphi(H)'$ \emph{cannot be bounded unless}
    $|\varphi'(\lambda)|\leq C\jap{\lambda}^{-2}$.  }\end{example}

\section{Commutators and decay}\label{s:decay}

>From Proposition \ref{pr:com} we get the following decay result.

\begin{proposition}\label{pr:decay1}
  Let $H\in C^1(A)$ such that $H':=[H,\i A]$ commutes with $H$ and let
  $u\in D(H)\cap D(A)$. Then $|\braket{u}{H'\e^{\i t H}u}|\leq
  2|t|^{-1}\|Au\|\|u\|$. In particular, if $H'=B^*B$ for some
  continuous $B:D(H)\to\ch$ commuting with $H$, then
  $|\psi_{Bu}(t)|\leq C_u\jap{t}^{-1}$ for $u\in D(H)\cap D(A)$. If $B$ is
  bounded on $\ch$ then this holds for all $u\in D(A)$.
\end{proposition}

\proof The first part is obvious. The fact that $B$ commutes with $H$
means $\e^{\i tH}B=B \e^{\i tH}$ for any $t$ and this clearly implies
that $[A,H]$ commutes with $H$. Then $\braket{Bu}{\e^{\i
    tH}Bu}=\braket{u}{[H,\i A]\e^{\i tH}u}$ hence the second and the
third part are consequences of the first one. \qed

\begin{remark}\label{re:classic}{\rm Some of the next results are
    abstract versions of the following estimate: if $H=h(q)$ and
    $A=-p$ in $L^2(\R)$ then $H'=h'(q)$ hence if $|h'|\geq c>0$ and
    $h''/h'^2$ is bounded then an integration by parts gives
    $\left|\int\e^{\i th(x)}|u(x)|^2\d x \right|\leq C_u \jap{t}^{-1}$
    if $u\in D(p)$.  }\end{remark}

We shall say that a densely defined operator $S$ on $\ch$ is
\emph{boundedly invertible} if $S$ is injective, its range is dense,
and its inverse extends to a continuous operator on $\ch$. If $S$ is
symmetric this means that $S$ is essentially self-adjoint and $0$ is
in the resolvent set of its closure.

\begin{proposition}\label{pr:decay2}
  Let $H\in C^1(A)$ such that $H'$ commutes with $H$ and
  $H'D(H)\subset\ch$. Assume that $H'$, when considered as operator on
  $\ch$, is boundedly invertible and $H'^{-1}$ extends to a bounded
  operator of class $C^1(A)$.  Then $|\psi_u(t)|\leq C_u\jap{t}^{-1}$
  if $u\in D(A)$.
\end{proposition}

\proof From Proposition \ref{pr:com} we get
$[\e^{\i tH},A] =tH'\e^{\i tH}$ as operators $D(H)\to D(H)^*$ hence
$[\e^{\i tH},A]$ is a bounded operator $D(H)\to\ch$ and we have
$[\e^{\i tH},A] H'^{-1}=t\e^{\i tH}$ on the range of $H'$. We denote
$K$ the continuous extension to $\ch$ of $H'^{-1}$ and note that $K$
commutes with $H$ because $H'\e^{\i tH}=\e^{\i tH}H'$ hence
$H'^{-1}\e^{\i tH}=\e^{\i tH}H'^{-1}$ for all $t$. If $u\in D(A)$ and
$Ku\in D(H)$ then $Ku\in D(A)$ because $K\in C^1(A)$ hence
\[
t\psi_u(t)=\braket{u}{[\e^{\i tH},A] Ku} =\braket{\e^{-\i tH}u}{A Ku}
-\braket{Au}{\e^{\i tH} Ku}  .
\]
This implies
\[
|t\psi_u(t)|\leq \|u\|\|AKu\|+\|Au\|\|Ku\|
\leq \|[A,K]\|\|u\|^2+2 \|K\|\|u\|\|Au\| .
\]
Now let $u$
be an arbitrary element of $D(A)$. Let $R_\varepsilon=(1+\i\varepsilon
H)^{-1}$ and $u_\varepsilon=R_\varepsilon u$. Then $u_\varepsilon\in
D(A)$ because $R_\varepsilon D(A)\subset D(A)$ and $Ku_\varepsilon
=R_\varepsilon Ku\in D(H)$ hence we have
\[
|t\psi_{u_\varepsilon}(t)| \leq \|[A,K]\|\|u_\varepsilon\|^2+2
\|K\|\|u_\varepsilon\|\|Au_\varepsilon\| .
\]
Since $[A,R_\varepsilon]=\varepsilon H' R_\varepsilon^2$ we
get $|t\psi_u(t)|\leq \|[A,K]\|\|u\|^2+2 \|K\|\|u\|\|Au\|$
by making $\varepsilon\to0$ in the preceding inequality. \qed

\begin{remark}\label{re:ess}{\rm
We may restate the assumptions of  Proposition \ref{pr:decay2} as
follows: $H$ is of class $C^2(A)$, $H'D(H)\subset\ch$, and $H'$ when
seen as operator on $\ch$ is essentially self-adjoint and $0$ is not
in its spectrum.
}\end{remark}

\begin{remark}\label{re:0}{\rm The good decay
    $\psi_u(t)=\mathcal{O}(t^{-1})$ obtained in Proposition \ref{pr:decay2}
    depends on a quite strong condition on $H'$ which in particular
    forces $H'$ to be an essentially self-adjoint operator on $\ch$
    whose spectrum does not contain zero. In the ``classical'' case
    mentioned in the Remark \ref{re:classic} this means $|h'(x)|\geq
    c>0$ which is rather natural when one has to estimate an integral
    like $\psi(t)=\int \e^{\i th(x)}f(x) \d x$ for large positive $t$:
    \emph{points of stationary phase should be avoided, otherwise we
      cannot expect more than $\psi(t)=\mathcal{O}(t^{-1/2})$.}
  }\end{remark}

We now consider operators satisfying some special commutation
relations but allow $H'$ to have zeros, e.g.\ we treat the simplest
case $H'=cH$. Note that Example \ref{ex:exo} shows that requiring only an algebraic
relation like $[H,\i A]=c H$ is highly ambiguous; the property $H\in
C^1(A)$ is then necessary and is not automatically satisfied.

In many of the applications of the conjugate operator method, see for
example Section \ref{s:apps}, the operator $A$ is unbounded in energy
space. However, it is possible to introduce an energy cut-off for $A$
that does not alter the $C^1(A)$ condition for $H$ and preserves the
behaviour of the commutation relation at thresholds. For instance,
consider $H\in C^1(A)$ bounded from below and such that
$H'=cH$. Define the operators $g(H)=(H+c)^{-1/4}$ and
$\tilde{A}=g(H)Ag(H)$. Then, it is easy to see that
$H\in C^1(\tilde{A})$ and $[H,\i\tilde A]=H'g(H)^2$.

\begin{remark}\label{re:projection} {\rm The subsequent results will
    hold for self evanescent states $u\in\ce$, but they also rely on
    the condition $Au\in\ce$. The latter assumption is not satisfied
    in general, in fact, it is implied by a stronger localization
    condition for $u$. To elude this, it can be assumed that there is
    a projection $P$ which commutes with $H$, and such that
    $u\in\text{Ran}P$. Then the condition $Au\in\ce$ can be replaced
    by $PAu\in\ce$, which is easier to satisfy. This idea will be explored further in forthcoming work. For instance, one can choose $P$ as the projection on the continuous spectrum of $H$ and the proofs presented below can be slightly modified to obtain the same decay
    estimates.}
\end{remark}

\begin{proposition}\label{pr:decay3}
  Let $H\in C^1(A)$ such that $H'=cH$ with $c\in\R\setminus\{0\}$ and
  let $u\in D(A)$ such that $u,Au\in\ce$. Then
  $|\psi_u(t)|\leq C_u\jap{t}^{-1/2}$.
\end{proposition}

\proof
We have $\psi_u\in L^2(\R)$ because $u\in\ce$ hence, according to
Corollary \ref{co:psidecay}, it suffices to show that the function
$(\delta\psi)(t)=t\psi_u'(t)$ also belongs to $L^2(\R)$. If $u\in
D(|H|^{1/2})$ then $t\psi_u'(t)=\braket{u}{\i t H\e^{\i tH}u}$ so that
by using Proposition \ref{pr:com} we get:
\[
\i ct\psi_u'(t)=\braket{u}{t cH\e^{\i tH}u}
=\braket{u}{[A,\i tH ]\e^{\i tH}u}
=\braket{u}{[A,\e^{\i tH}]u}  .
\]
Then, if $u\in D(A)$ we get
$\i ct\psi_u'(t) =\braket{Au}{\e^{\i tH}u} - \braket{u}{\e^{\i tH}Au}$
hence \eqref{eq:ineq} implies:
\begin{equation}\label{eq:estim}
c^2\|\delta\psi_u\|^2 \leq 2\|\psi_u\|_{L^2}\|\psi_{Au}\|_{L^2}
=2 [u]^2_H [Au]^2_H .
\end{equation}
So the proposition is proved under the supplementary condition $u\in
D(H)$ and the estimate \eqref{eq:estim} depends only on $c$.
Now consider an arbitrary $u\in D(A)$ such that $u,Au\in\ce$ and for
$\varepsilon>0$ let $R_\varepsilon=(1+\i\varepsilon H)^{-1}$.  Then
from Proposition \ref{pr:abg} we get $R_\varepsilon u\in D(A)$ and
$[A,R_\varepsilon] = R_\varepsilon [\i\varepsilon H,A]R_\varepsilon =
c\varepsilon HR^2_\varepsilon$. If we set $u_\varepsilon=R_\varepsilon
u$ then the estimate \eqref{eq:estim} gives
$c^2\|\delta\psi_{u_\varepsilon}\|^2 \leq 2 [u_\varepsilon]^2_H
[Au_\varepsilon]^2_H$. Finally, let $\varepsilon\to0$ and use Fatou's
lemma in the left hand side and Lemma \ref{lm:scom} on the right hand
side to get \eqref{eq:estim} without the condition $u\in D(H)$.  \qed

\begin{theorem}\label{th:decay5}
  Let $H\in C^1(A)$ such that $H'=\theta(H)$ with $\theta$ real of
  class $C^1$ with bounded derivative and such that: (1) if
  $|\lambda|\geq\varepsilon>0$ then $|\theta(\lambda)|\geq
  c_\varepsilon>0$, (2) $\lambda/\theta(\lambda)$ extends to a $C^1$
  function on $\R$.  If $u\in D(A)$ and $u,Au\in\ce$ then
  $|\psi_u(t)|\leq C_u\jap{t}^{-1/2}$.
\end{theorem}

\proof Let $\varphi\in C^\infty_\rmc(\R)$ real and equal to one on a
neighbourhood of zero and let us set
$\phi=\varphi(H),\phi^\perp=1-\phi^2$, so that $\psi_u(t)=\psi_{\phi
  u}(t)+\braket{u}{\phi^\perp\e^{\i tH}u}$. We first show that the
second term is $\mathcal{O}(t^{-1})$. We have $\phi^\perp H'^{-1}=\xi(H)$ with
$\xi(\lambda)=(1-\varphi^2(\lambda))/\theta(\lambda)$ hence
\[
t \braket{u}{\phi^\perp\e^{\i tH}u}
=\braket{\phi^\perp H'^{-1}u}{tH'\e^{\i tH}u}
=\braket{\xi(H)u}{tH'\e^{\i tH}u}=\braket{\xi(H)u}{[\e^{\i tH},A]u}.
\]
Until here $u$ was an arbitrary element of $\ch$. If $u\in D(H)\cap
D(A)$ then we can expand the commutator and get
\begin{align*}
|t \braket{u}{\phi^\perp\e^{\i tH}u}|
&=|\braket{\e^{-\i tH}\xi(H)u}{Au}
- \braket {A\xi(H)u}{\e^{\i{}tH}u}|\\
&\leq \|Au\|\|\xi(H)u\|+\|A\xi(H)u\|\|u\|.
\end{align*}
Since $\xi$ is a bounded function of class $C^1$ with bounded
derivative we can use Proposition \ref{pr:dense} and get
$\xi(H)'=\xi'(H)H'=(\xi'\theta)(H)$. We have
$\xi'\theta=-\theta'/\theta$ outside a compact neighbourhood of zero,
hence $\xi(H)'$ is a bounded operator, so $\xi(H)$ is of class
$C^1(A)$, hence $\xi(H)D(A)\subset D(A)$. Then, since
$H'D(H)\subset\ch$, this estimate remains true for any $u\in D(A)$
by Proposition \ref{pr:dense}. Thus $|\braket{u}{\phi^\perp\e^{\i tH}u}|\leq C_u\jap{t}^{-1}$ for
any $u\in D(A)$.

>From now on we change notations: $\phi u$ will be denoted $u$. So we
may assume $\supp E_u\subset[-1,1]$ and $u,Au\in\ce$, cf.\ Lemma
\ref{lm:scom}, and we want to prove that the function $t\psi_u'(t)$
belongs to $L^2(\R)$. Let $\eta$ be the $C^1$ function on $\R$ which
extends $\lambda/\theta(\lambda)$, let $\zeta\in C^\infty_\rmc(\R)$
such that $\zeta(H)u=u$, and let us set $\tilde\eta=\eta\zeta$. Then
$\tilde\eta(H)u \in D(A)\cap\ce$ and
$A\tilde\eta(H)u=[A,\tilde\eta(H)]u+\tilde\eta(H)Au\in \ce$. Finally
\begin{align*}
-\i t\psi_u'(t) & =\braket{u}{tH\e^{\i tH}u}
=\braket{u}{\eta(H) tH'\e^{\i tH}u}
=\braket{\tilde\eta(H) u}{ [\e^{\i tH},A]\zeta(H) u} \\
& =\braket{\e^{-\i tH}\tilde\eta(H) u}{A\zeta(H)u}
-\braket{A\tilde\eta(H) u}{\e^{\i tH}\zeta(H)u}
\end{align*}
and from \eqref{eq:ineq} we get the square integrability of
$t\psi_u'(t)$. \qed

\begin{example}\label{ex:decay4}
  Let $H\in C^1(A)$ such that $H\geq0$ and $H'=H(1+H)^{-1}$. If $u\in
  D(A)$ and $u,Au\in\ce$ then $|\psi_u(t)|\leq C_u\jap{t}^{-1/2}$.
\end{example}

\section{Higher-order commutators}\label{s:calpha}

The decay estimates obtained so far on $\psi_u(t)$ are at most of
order $\mathcal{O}(t^{-1})$ and it is clear that to obtain  $\mathcal{O}(t^{-k})$ for some
integer $k>1$ we need conditions of the form $u\in D(A^k)$ and
assumptions on the higher-order commutators of $A$ with $H$. We
recall here the necessary formalism.

Let $A$ be a self-adjoint operator on a Hilbert space $\ch$ and
$k\in\N$.  We say that \emph{$S$ is of class $C^k(A)$}, and we write
$S\in C^k(A)$, if the map
$ \R\ni t\mapsto \e^{-\i t A}S\e^{\i tA}S\in B(\ch)$ is of class $C^k$
in the strong operator topology. It is clear that $S\in C^{k+1}(A)$ if
and only if $S\in C^{1}(A)$ and $S'\in C^{k}(A)$. If $S\in C^{2}(A)$
we set $(S')'=S''=S^{(2)}$, etc.  Clearly $C^{k}(A)$ is a
$*$-subalgebra of $ B(\ch)$ and if $S\in B(\ch)$ is bijective and
$S\in C^{k}(A)$ then $S^{-1}\in C^{k}(A)$.

For any $S\in B(\ch)$ let $\ca(S)=[S,\i A]$ considered as a
sesquilinear form on $D(A)$. We may iterate this and define a
sesquilinear form on $D(A^k)$ by:
\[
S^{(k)}\equiv\ca^k(S)=\i^k\sum_{i+j=k} \frac{k!}{i!j!} (-A)^i S A^j.
\]
Then $S\in C^k(A)$ if and only if this form is continuous for the
topology induced by $\ch$ on $D(A^k)$. We keep the notation $\ca^k(S)$
or $S^{(k)}$ for the bounded operator associated to its continuous
extension to $\ch$.

Strictly speaking, the operator $\ca$ acting in $B(\ch)$ must be
defined as the infinitesimal generator of the group of automorphisms
$\cu=\{\cu_t\}_{t\in\R}$ of $B(\ch)$ given by
$\cu_t(S)\equiv \e^{t\ca}(S)=\e^{-\i t A}S\e^{\i tA}$. This group is
not of class $C_0$ and so $\ca$ is not densely defined. Then $C^k(A)$
is just the domain of $\ca^k$. One may also define $C^\alpha(A)$ if
$\alpha$ is not an integer as the Besov space of order
$(\alpha,\infty)$ associated to $\cu$.

We denote $B_1(\ch)$ the Banach algebra of trace class operators on
$\ch$. Its dual is identified with the space $B(\ch)$ of all bounded
operators on $\ch$ with the help of the bilinear form $\Tr(S\rho)$. It
is clear that the restrictions of the $\cu_t$ to
$B_1(\ch)\subset B(\ch)$ give a group of automorphisms of $B_1(\ch)$
and that this group is of class $C_0$. We do not distinguish in
notation between $\cu$ and $\ca$ and their restrictions to $B_1(\ch)$
but note that for example the domain of $\ca$ in $B_1(\ch)$ is the set
of $S\in C^1(A)\cap B_1(\ch)$ such that $\ca(S)\in B_1(\ch)$.
Moreover, if $S=\ket{u}\bra{v}$ and $u,v\in D(A^k)$ then $S$ belongs
to the domain of $\ca^k$ in $B_1(\ch)$.

Now let $H$ be a self-adjoint operator on $\ch$ and
$ R(z)= (H-z)^{-1}$ for $z$ in the resolvent set $\rho(H)$ of $H$.  We
say that \emph{$H$ is of class $C^k(A)$} if $R(z_{0})\in C^k(A)$ for
some $z_{0}\in\rho(H)$; then we shall have $R(z)\in C^k(A)$ for all
$ z\in \rho(H)$ and more generally $\varphi(H)\in C^k(A)$ for a large
class of functions $\varphi$ (e.g. rational and bounded on the
spectrum of $H$).

For each real $m$ let $S^m(\R)$ be the set of symbols of class $m$ on
$\R$, i.e.\ the set of functions $\varphi:\R\to\C$ of class $C^\infty$
such that $|\varphi^{(k)}(\lambda)|\leq C_k\jap{\lambda}^{m-k}$ for
all $k\in\N$. Note that $S^m\cdot S^n\subset S^{m+n}$ and
$\varphi^{(j)}\in S^{m-j}$ if $\varphi\in S^{m}$ and $j\in\N$.

\begin{proposition}\label{pr:kcom}
  Let $H$ be a self-adjoint operator of class $C^1(A)$ with
  $H'=\theta(H)$ for some $\theta\in S^2(\R)$.  Then $H$ is of class
  $C^\infty(A)$. Let $\delta_\theta$ be the first order differential
  operator given by
  $\delta_\theta=\theta(\lambda)\frac{d}{d\lambda}$. If
  $\theta\in S^{1}(\R)$ and $\varphi\in S^{0}(\R)$ then $\varphi(H)$
  is of class $C^\infty(A)$ and
\begin{equation}\label{eq:kderiv}
\ca^k \left( \varphi(H) \right) = \left( \delta_\theta^k\varphi
\right)(H) \ \forall k\in\N.
\end{equation}
\end{proposition}

\proof We begin with a general remark. By using Proposition
\ref{pr:com} we see that if $H$ is of class $C^1(A)$ and
$H'=\theta(H)$ for some real Borel function $\theta$, and if
$\varphi\in C^1_\rmb(\R)$, then
$\varphi(H)' = \ca(\varphi(H)) = \theta(H)\varphi'(H) =
(\delta_\theta\varphi)(H)$.
In particular, if $\delta_\theta\varphi=\theta\varphi'$ is a bounded
function then $\varphi(H)$ is of class $C^1(A)$.

If we take $\theta\in S^2$ and $\varphi(\lambda)=(\lambda+\i)^{-1}$
then $\varphi\in S^{-1}$ hence $\theta\vphi'\in S^0$. Thus the
operator $R=(H+\i)^{-1}=\varphi(H)$ satisfies $R'=\psi(H)$ with
$\psi\in S^0$. Now we may apply the preceding argument with $\vphi$
replaced by $\psi$ and get $\psi\in C^1(A)$, so $R'\in C^1(A)$,
etc. This proves that $H$ is of class $C^\infty(A)$.

In the preceding argument we clearly may take any $\varphi\in S^{-1}$.
If $\theta\in S^1$ then the same argument works for any
$\varphi\in S^0$ and gives the last assertion of the proposition. \qed

\begin{remark}\label{re:m}{\rm If $\theta\in S^m$ and
    $\varphi\in S^{-(m-1)}$ with $1\leq m\leq 2$ the last assertion of
    the proposition remains true (with the same proof). }\end{remark}

We finish this section with some comments in connection with relation
\eqref{eq:kderiv}. At a formal level \eqref{eq:kderiv} means
\begin{equation}\label{eq:kderivf}
  \e^{-\i tA}\varphi(H)\e^{\i tA} \equiv \e^{t\ca} \left( \varphi(H) \right)
= \left( \e^{t\delta_\theta}\varphi \right)(H) .
\end{equation}
We shall explain without going into details how one may rigorously
interpret this relation and how one may use it to get decay
estimates.

Let $\xi_t$ be the flow of diffeomorphisms of the real line defined by
the vector field
$\delta_\theta=\theta(\lambda)\frac{d}{d\lambda}$. This means that
$\frac{d}{dt}\xi_t(\lambda)=\theta(\xi_t(\lambda))$ and
$\xi_0(\lambda)=\lambda$ for all $\lambda\in\R$ (we assume that such a
global flow exists). Then if $\varphi:\R\to\C$ is a smooth function we
have $\frac{d}{dt}\varphi\circ\xi_t=(\delta_\theta\varphi)\circ\xi_t$
or $\varphi\circ\xi_t=\e^{t\delta_\theta}\varphi$. Hence
\eqref{eq:kderivf} may be written
$\e^{-\i tA}\varphi(H)\e^{\i tA} = (\varphi\circ\xi_t)(H)$.  This can
be easily checked independently of what we have done before.

Let $M(\R)$ be the space of all bounded Borel measures on $\R$.  We
associate to $H$ a continuous linear map $\Phi:B_1(\ch)\to M(\R)$
defined as follows: if $\rho\in B_1(\ch)$ then
$\int \varphi\Phi(\rho)=\Tr(\varphi(H)\rho)$ for any bounded Borel function
$\varphi$. Then
\[
\Tr(\varphi(H)\cu_{-t}(\rho))=\Tr(\e^{-\i tA}\varphi(H)\e^{\i tA}\rho)
=\Tr((\varphi\circ\xi_t)(H)\rho)
\]
which means that the measure $\Phi(\cu_{-t}(\rho))$ is equal to the image
of the measure $\Phi(\rho)$ through the map $\xi_t$. Or, if we denote
$V_t$ the map $M(\R)\to M(\R)$ which sends a measure $\mu$ into its
image $\xi_t^*(\mu)$ through $\xi_t$, we have
$\Phi\circ\cu_{-t}=V_t\circ\Phi$.

Thus, if $\rho$ belongs to the Besov space $B_1(\ch)_{s,p}$ associated
to the group of automorphisms $\cu_t$ of $B_1(\ch)$ then $\Phi(\rho)$
belongs to the Besov space $M(\R)_{s,p}$ associated to the group of
automorphisms $V_t$ of $M(\R)$ (notations as in \cite{ABG}).  This
gives smoothness properties of the measure $\Phi(\rho)$ with respect
to the differential operator $\delta_\theta$ in terms of smoothness
properties of $\rho$ with respect to the operator $A$.  In particular,
since $\Tr(\e^{\i tH}\rho)=\int \e^{\i t\lambda}\Phi_\rho(\d\lambda)$
is just the Fourier transform of the measure
$\Phi_\rho\equiv\Phi(\rho)$, this allows us to control the decay as
$t\to\infty$ of $t\mapsto\Tr(\e^{\i tH}\rho)$ in terms of the local
behaviour of the measure $\Phi_\rho$.

The operators $V_t$ can be explicitly computed in many situations and
the preceding strategy gives optimal results. For example, in the
simplest case $[H,\i A]=1$ we get for any $s>0$
\begin{equation}\label{eq:1}
\|\jap{A}^{-s}\e^{\i tH}\jap{A}^{-s}\|\leq C_s\jap{t}^{-s}
\end{equation}
If $[H,\i A]=H$ then such a good decay is impossible because zero is a
threshold (see Remark \ref{re:count}).  On the other hand, if $\eta$
is a smooth function equal to zero near zero and to one near infinity
then (see \cite{BGS}):
\begin{equation}\label{eq:2}
  \|\jap{A}^{-s}\e^{\i tH}\eta(H)\jap{A}^{-s}\|\leq C_s\jap{t}^{-s} .
\end{equation}
Estimates of this nature hold in fact for a large class of commutation
relations $[H,\rmi A]=\theta(H)$.  Moreover, if the function $\eta$ is
of compact support and such that a strict Mourre estimate holds on a
neighbourhood of its support then
$\jap{A}^{-s}\e^{\i tH}\eta(H)\jap{A}^{-s}$ may be controlled in terms
of the regularity of the boundary values of the resolvent
$R(\lambda\pm\rmi0)$ via a Fourier transformation argument. The higher-order continuity properties of the operators $R(\lambda\pm\rmi0)$ as
functions of $\lambda$ in a region where one has a strict Mourre
estimate has been studied by commutator methods first in \cite{JMP}
and then in \cite{BouG} where the optimal regularity result has been
obtained. This gives the following decay (see \cite{BGS}): if the
self-adjoint operator $H$ has a spectral gap and is of class
$C^{s+1/2}(A)$ for some real $s>1/2$, and if $\eta$ is a $C^\infty$
function with compact support in an open set where $A$ is strictly
conjugate to $H$, then there is a number $C$ such that
\begin{equation}\label{eq:hor}
\|\jap{A}^{-s}\e^{\i tH}\eta(H)\jap{A}^{-s}\|\leq C_s\jap{t}^{-(s-1/2)} \,.
\end{equation}
This decay is the best possible for $H$ of class $C^{s+1/2}(A)$. This may
be compared with the corresponding result in \cite[p.\ 222]{JMP} but
one should take into account the remark in \cite[p.\ 13]{BouG}. Note
that \eqref{eq:hor} is an endpoint estimate and it can be improved by
interpolation at intermediary points. For example, if
$H\in C^\infty(A)$ and $s>\varepsilon>0$ then we have
\begin{equation*}
\|\jap{A}^{-s}\e^{\i tH}\eta(H)\jap{A}^{-s}\|\leq
C_{s,\varepsilon}\jap{t}^{-(s-\veps)} \,.
\end{equation*}
The problem with these estimates is that the cutoff function $\eta$
eliminates the thresholds of $H$, i.e. exactly the energies in which
we are interested.  We have explained before that the behaviour of
$\|\jap{A}^{-s}\e^{\i tH}\jap{A}^{-s}\|$ may be very bad because of
the thresholds.

We emphasize that in this article we are mainly interested in global
estimates which take into account the existence of thresholds. In
order to get some decay we consider self evanescent states
$u\in\ce(H)$ and show how one can get a better decay of the physically
meaningful quantity $\psi_u(t)$.  Our results are obtained by a direct
study of the evolution operator $\rme^{\rmi tH}$ and do not involve
regularity properties of the resolvent of $H$.

\section{Higher-order decay}\label{s:hod}

The expressions $\psi_u(t)=\braket{u}{\e^{\i tH}u}$ that we considered
until now are quadratic in $u$ and this complicates the computations of higher order.  To elude this we note that
$\psi_u(t)=\Tr(\e^{\i tH} \rho)$ with $\rho=\ket{u}\bra{u}$,
expression which makes sense for any $\rho\in B_1(\ch)$ and is linear
in $\rho$.

We begin with an extension to higher orders of Proposition
\ref{pr:decay2}.

\begin{theorem}\label{th:kbdd}
  Let $k\in\N$ and $s\in[0, k]$ real. Assume that $H$ is of class
  $C^{k+1}(A)$ and $H'$ commutes with $H$, satisfies
  $H'D(H)\subset\ch$, and is boundedly invertible. Then for each
  vector $u\in D(|A|^s)$ we have $\psi_u(t)=\mathcal{O}(t^{-s})$.
\end{theorem}

\proof By an interpolation argument it suffices to prove
$|\psi_u(t)|\leq C_k(\|u\|+\|A^ku\|)^2\jap{t}^{-k}$ for $u$ in a dense
subspace of $D(A^k)$.  Formally this is quite straightforward starting
with the formula $(\i tH')^{-1}\ca(\e^{\i tH})=\e^{\i tH}$ and then
iterating it $k$ times; we next sketch the rigorous proof. We change
slightly the notations from the proof of Proposition \ref{pr:decay2}
and denote $K$ the continuous extension to $\ch$ of $-\i H'^{-1}$.
Then $K$ commutes with $H$, is of class $C^k(A)$, and we have
$K\ca(\e^{\i tH})=\ca(\e^{\i tH})K=t\e^{\i tH}$.  Let $u\in D(A^k)$
and $\rho=\ket{u}\bra{u}$ or a more general trace class operator. Let
$L_K$ and $R_K$ be the operators of right and left multiplication by
$K$, which act both in $B(\ch)$ and in $B_1(\ch)$.  Then
$R_K\ca(\e^{\i tH})=t\e^{\i tH}$ hence
\begin{align*}
t\psi_u(t) &=\Tr\left(( R_K\ca)(\e^{\i tH}) \rho \right)
=\Tr\left(\ca(\e^{\i tH})(K\rho) \right) \\
&=-\Tr\left(\e^{\i tH}\ca(K\rho) \right)
=-\Tr\left(\e^{\i tH}(\ca L_K)\rho \right)
\end{align*}
This is easy to justify since $Ku \in D(A^k)$ because $K$ is of class
$C^k(A)$. In exactly the same way, starting with $( R_K\ca)^k(\e^{\i
  tH})=t^k\e^{\i tH}$ we get
\[
t^k\psi_u(t) =\Tr\left(( R_K\ca)^k(\e^{\i tH}) \rho \right)
=(-1)^k\Tr\left(\e^{\i tH}(\ca L_K)^k\rho \right) .
\]
Finally, it remains to note that
$\|(\ca L_K)^k\rho\|_{B_1(\ch)}\leq C_k(\|u\|+\|A^ku\|)^2$. \qed

\begin{remark}\label{re:count}{\rm The following example shows that
    such a good decay as in Theorem \ref{th:kbdd} cannot be expected
    if $H'$ is not boundedly invertible. In the Hilbert space
    $\ch=L^2(0,\infty)$ let $H$ be the operator of multiplication by
    the independent variable $x$ and let $A$ be the self-adjoint
    realization of $\frac{\i}{2}(x\frac{d}{dx}+\frac{d}{dx}x)$.  Then
    $H$ is of class $C^\infty(A)$ and $H'=[H,\i A]=H$. Let $u$ be a
    $C^\infty$ function on $(0,\infty)$ which is zero for $x>2$ and
    equal to $x^{-\theta}$ for $x<1$ with $0<\theta<1/2$. Then
    $u\in D(|A|^s)$ for all $s>0$ but
    $\psi_u(t)\sim\int_0^1\e^{\i tx}x^{-2\theta}\d x \sim
    t^{2\theta-1}$
    for $t\to\infty$, hence the decay can be made as bad as possible.
    On the other hand, Example \ref{ex:eva} explains why the space
    $\ce$ helps to improve the behaviour.  }\end{remark}

We now give a higher-order version of Theorem \ref{th:decay5}. Recall
that $\theta\in S^m(\R)$ is an \emph{elliptic symbol} if there is
$c>0$ such that $|\theta(\lambda)|\geq c |\lambda|^m$ near
infinity. Then $\eta/\theta\in S^{-m}(\R)$ for any $C^\infty$ function
$\eta$ with support in the region where $\theta\neq0$ and equal to one
near infinity.

\begin{theorem}\label{th:kzero}
  Let $H\in C^{1}(A)$ such that $H'=\theta(H)$ for some
  elliptic symbol $\theta\in S^m(\R)$ with $0\leq m\leq 1$. Assume:
  (1) $\theta(\lambda)\neq0$ if $\lambda\neq0$ and (2)
  $\lambda/\theta(\lambda)$ extends to a $C^\infty$ function on $\R$.
  Let $k$ be an odd integer and let $u\in\ch$ be of the form
  $|H|^{(k-1)/4}v$ for some $v\in D(A^k)$ such that $A^jv\in\ce$ for
  $0\leq j\leq k$. Then $|\psi_u(t)|\leq C_u\jap{t}^{-k/2}$.
\end{theorem}

\proof
Denote $S^0_{(0)}(\R)$ the set of $a\in S^0(\R)$ such that
$a(\lambda)=0$ near zero.  We first prove the following: if $n\in\N$
and $a\in S^0_{(0)}(\R)$ then there are $a_0,a_1,\dots,a_n\in S^0(\R)$
such that:
\begin{equation}\label{eq:ind}
t^na(H)\e^{\i tH}=\sum_{j=0}^n \ca^j\left( a_j(H) \e^{\i tH} \right)  .
\end{equation}
Of course, the $a_j$ also depend on $n$.  If $n=1$ we write (see also the
proof of Theorem \ref{th:decay5}):
\begin{equation}\label{eq:in1}
ta(H)\e^{\i tH}= -\i \frac{a(H)}{\theta(H)}\ca\left(\e^{\i tH}
\right)=\ca\left(a_1(H)\e^{\i tH} \right)+a_0(H) \e^{\i tH}
\end{equation}
where $a_1=\frac{a}{\i\theta}$ and $a_0=-\theta a_1'$ (use Proposition
\ref{pr:kcom}). We mention that we use without comment the relation
$\ca(ST)=\ca(S)T+S\ca(T)$ with the further simplification that in our
context $S$ and $T$ are functions of $H$ hence commute.  Now assume
\eqref{eq:ind} is true and let us prove it with $n$ replaced by
$n+1$. Let $b\in C^\infty$ equal to zero near zero and to $1$ near
infinity and such that $a_j=a_jb$ for all $j$. Then
\[
t^{n+1}a(H)\e^{\i tH}=\sum_{j=0}^n \ca^j\left( a_j(H) tb(H)\e^{\i tH} \right)
\]
Now we use \eqref{eq:in1} and replace
$tb(H)\e^{\i tH}=\ca\left(b_1(H)\e^{\i tH} \right)+b_0(H) \e^{\i
  tH}$. Thus
\begin{align*}
  a_j(H) tb(H)\e^{\i tH} &=a_j(H) \ca\left(b_1(H)\e^{\i tH}\right)
+a_j(H) b_0(H) \e^{\i tH} \\
&=\ca\left(a_j(H)b_1(H)\e^{\i tH}\right)
+\left(a_j(H) b_0(H) - \ca(a_j(H))b_1(H)\right)\e^{\i tH}
\end{align*}
which clearly gives the required result.

Now we begin the proof of the theorem. As in the proof of Theorem
\ref{th:decay5} we consider separately the case when $u$ is zero near
energy zero and that when $u=\varphi(H)u$ for some $\varphi\in
C_\rmc^\infty$. The first case is an immediate consequence of
\eqref{eq:ind} because there is $a\in S^0_{(0)}(\R)$ such that
$a(H)u=u$ hence (recall the notation $\rho=\ket{u}\bra{u}$)
\begin{equation}\label{eq:triv}
t^k\braket{u}{\e^{\i tH}u}=\Tr\left(t^ka(H)\e^{\i tH}  \rho\right)
=\sum_{j=0}^k (-1)^j \Tr\left(a_j(H) \e^{\i tH} \ca^j(\rho)\right)
\end{equation}
which implies $\psi_u(t)=\mathcal{O}(t^{-k})$ because obviously
$\rho\in D(\ca^k)$ if $u\in D(A^k)$.

Note that the facts established above hold for an arbitrary
$u\in\ch$. The condition involving $v$ is needed to have some control
on the behavior of $u$ at zero energy, which cannot be arbitrary as
explained in Remark \ref{re:count}. When we localize near zero energy
we replace $u$ by $\varphi(H)u$ with $\varphi\in C^\infty_\rmc$ equal
to one on a neighbourhood of zero. If $u=|H|^m v$ with $m=(k-1)/4$ and
$v\in D(A^k)$ such that $A^jv\in\ce$ for $0\leq j\leq k$ then
$\varphi(H)u=|H|^m \varphi(H)v$.  By Proposition \ref{pr:kcom} $H$ is
of class $C^\infty(A)$ so $\varphi(H)D(A^j)\subset D(A^j)$ for any $j$
and $A^j\varphi(H)v\in\ce$ by Lemma \ref{lm:scom}.

Thus for the rest of the proof we may assume that the support of $u$
in a spectral representation of $H$ is included in $[-1,1]$ and
$u=|H|^m v$ with $v\in D(A^k)$ such that $A^jv\in\ce$ for
$0\leq j\leq k$.  It is clear that $v$ has the same $H$-support as
$u$.  Our purpose is to check the assumptions of the Corollary
\ref{co:kdecay} for $\psi=\psi_u$.  There are two conditions to be
verified: the functions $t^{\frac{k-1}{2}}\psi_u(t)$ and
$t^{\frac{k+1}{2}}\psi'_u(t)$ should be in $L^2(\R)$. We treat only
the second one, the first is treated similarly.  If
$\ell=2m+1=(k+1)/2$ then
\[
t^\ell\psi_u'(t)=\braket{u}{\i t^\ell H \e^{\i tH}u}
=\braket{|H|^mv}{\i t^\ell H \e^{\i tH}|H|^mv}
=\braket{v}{\i t^\ell H^\ell{\rm sgn}^{\ell+1}(H) \e^{\i tH}v}.
\]
Let $\eta$ be a $C^\infty$ function with compact support such that
$\eta(\lambda)=\lambda/\theta(\lambda)$ on $[-1,1]$. Then
$\lambda=\eta(\lambda)\theta(\lambda)$ on $[-1,1]\setminus\{0\}$ hence
on $[-1,1]$ so we have
\[
\i^{\ell-1} t^\ell\psi_u'(t) =\braket{\eta(H)^\ell v}{\i^\ell t^\ell \theta(H)^\ell
  \e^{\i tH}v} =\braket{\eta(H)^\ell v}{(\i tH')^\ell \e^{\i tH}v}.
\]
Recall that we have $\ca\left(\e^{\i tH}\right)=\i tH'\e^{\i tH}$ in a
sense described in Proposition \ref{pr:com}. But under the present
conditions we have much more because $H'D(H)\subset\ch$ hence
$\e^{\i\tau A}$ leaves invariant the domain of $H$ and induces there a
$C_0$-group (see the assertion (2) page \pageref{p:2}). In particular,
the set of $u\in D(H)\cap D(A^j)$ such that $A^ju\in D(H)$ for any
$j\in\N$ is dense in $D(H)$ (and is a core for $A$). Moreover, the
$\ca^j(H)$ are bounded operators if $j\geq2$.  This allows us to
compute $\ca^\ell\left(\e^{\i tH}\right)$ inductively as usual.  Our
next computations look slightly formal but it is straightforward,
although a little tedious, to rigorously justify each step.

Above we fixed $\ell$ to the value $(k+1)/2$ but now we allow it to
take any value smaller than this one. For the case $\ell=1$ see the
proof of Theorem \ref{th:decay5}. For $\ell=2$ we write
\begin{align*}
  \ca^2\left(\e^{\i tH}\right) &=\ca\left(\i tH'\e^{\i tH}\right)
=\i tH''\e^{\i tH} +(\i tH')^2\e^{\i tH} \\
&=\frac{H''}{H'} \i tH'\e^{\i tH} +(\i tH')^2\e^{\i tH}
=\frac{H''}{H'} \ca\left(\e^{\i tH}\right) +(\i tH')^2\e^{\i tH}
\end{align*}
By ``localising'' Proposition \ref{pr:kcom} we get
$H''=\ca(H')=\ca(\theta(H))=\theta(H)\theta'(H)=H'\theta'(H)$ hence
$\frac{H''}{H'}=\theta'(H)$. Thus
\[
(\i tH')^2\e^{\i tH} = \ca^2\left(\e^{\i tH}\right) -\theta'(H)
\ca\left(\e^{\i tH}\right) .
\]
Then if we set $\rho=\ket{\eta^2(H)v}\bra{v}$ we get
\begin{align*}
  \i t^2\psi_u'(t) & =\Tr\left( (\i tH')^2\e^{\i tH} \rho\right)
=\Tr\left( \ca^2\left(\e^{\i tH}\right) \rho\right)
  -\Tr\left( \theta'(H) \ca\left(\e^{\i tH}\right) \rho\right)   \\
& =\Tr\left( \e^{\i tH} \ca^2(\rho)\right)
  -\Tr\left(\e^{\i tH} \ca(\rho \theta'(H) ) \right).
\end{align*}
The right hand side belongs to $L^2(\R)$ by the argument from Theorem
\ref{th:decay5}, which finishes the proof in the case $\ell=2$.
The general case does not involve any new idea: by writing
conveniently $\frac{H^{(\ell)}}{H'}$ one may express
$(\i tH')^\ell \e^{\i tH}$ as a linear combination of functions of $H$
times commutators $\ca^j\left(\e^{\i tH}\right)$ and one may proceed
as above. \qed

\section{Applications}\label{s:apps}

We will use the previous results to obtain decay estimates for
$\psi_u(t)=\braket{u}{\e^{\i tH}u}$ in several situations. Note that
Example \ref{ex:exo} and Proposition \ref{pr:abg} show that the
commutation relation is not enough to prove the $C^1(A)$ condition for
$H$. For instance, in addition to the continuity of $[A,H]_0$ on
$D(A)\cap D(H)$, it suffices to verify the invariance of domain
$R(z)D(A)\subset D(A)$. In other cases it is convenient to verify the
simplified assumptions of Mourre \cite{Mo}, which are stronger than
the $C^1(A)$ property \cite{ABG}:

\begin{itemize}
\item[(a)]$e^{\i\theta A}D(H)\subset D(H)$
\item[(b)]There is a subspace $\mathscr{S}\subset D(A)\cap D(H)$ which
  is a core for $H$ such that
  $e^{\i\theta A}\mathscr{S}\subset\mathscr{S}$ and the form
  $[H,\i A]$ on $\mathscr{S}$ extends to a continuous operator
  $D(H)\to\ch$.
\end{itemize}

Recall the subspace $\ce=\{u\in\ch \mid [u]_H<\infty\}$, where
$[u]_H= \left(\int_\R |\psi_u(t)|^2 \d t \right)^{1/4}.$

\subsection{Example 1:\textnormal{ Laplacian in $\R^n$.}}$ $

Let $H=-\Delta$ in $L^2(\mathbb{R}^n)$ with domain the Sobolev space
$\ch^2(\mathbb{R}^n)$ and $A=-\tfrac{\i}{2}(x\cdot\nabla+\nabla\cdot
x)$ the generator of dilations which is essentially self-adjoint on
the Schwartz space
$\mathscr{S}=\mathcal{S}(\mathbb{R}^n)$. Condition (a) is a
consequence of the formula $e^{\i\theta
  A}(H+\i)^{-1}=(e^{-2\theta}H+\i)^{-1}e^{-\i\theta A},$ and (b) is
satisfied since $\mathscr{S}$ is a core for $H$ which is trivially
invariant under the dilation group.  Integration by parts on
$\mathscr{S}$ shows that $[H,\i A]=2H$.  We conclude from Proposition
\ref{pr:decay3} that for $u\in D(A)$ such that $u,Au\in\ce$,
$\psi_u$ satisfies the decay estimate $|\psi_u(t)|\leq
C_u\jap{t}^{-1/2}$. Higher-order decay estimates follow from Theorem
\ref{th:kzero}.

\subsection{Example 2:\textnormal{  $H=-\partial_{xx}+\partial_{yy}$ in $\R^2$.}}$ $

Let $H=-\partial_{xx}+\partial_{yy}$ and
$A=-\tfrac{\i}{2}(x\cdot\nabla+\nabla\cdot x)$ in $L^2(\R^2)$. With
the help of a Fourier transformation we see that $H$ is essentially
self-adjoint on $\mathcal{S}(\mathbb{R}^2)$. Clearly
$[H,\i A]=2H$, hence the estimate of Example 1
holds. One may treat similarly the case when the operator $H$ in
$L^2(\R^n)$ is an arbitrary homogeneous polynomial of order $m$ in the
derivatives $\i\partial_1,\dots,\i\partial_n$ with constant
coefficients: then $[H,\i A]=mH$.

\subsection{Example 3:\textnormal{ Electric field in $\R^n$.}}$ $

Here we study the case $H=-\Delta+\vec{h}\cdot x$ and
$A=\i\vec{h}\cdot\nabla$ in $\R^n$, where $\vec h$ is a fixed unitary
vector. We take again $\mathscr{S}=\mathcal{S}(\mathbb{R}^n)$ as a
core for $H$ and then it is easy to check the commutation relation
$[H,\i A]=1$. Therefore Proposition \ref{pr:decay1} provides the
estimate $|\psi_{u}(t)|\leq C_u\jap{t}^{-1}$ for $u\in D(A)$, where
$C_u=2\|u\|\|Au\|$. Further estimates follow from Theorem
\ref{th:kbdd}.

\subsection{Example 4:\textnormal{
    $H=-x^{2-\theta}\Delta-\Delta x^{2-\theta}$ in $\R_+$.}}$ $

For $0\,{<}\,\theta\,{<}\,2$ consider
$H=-x^{2-\theta}\Delta-\Delta x^{2-\theta}$ and
$A=-\tfrac{\i}{2}(x\cdot\nabla+\nabla\cdot x)$ in $\R_+$. Then
$\mathscr{S}=C^\infty_\rmc(\mathbb{R_+})$ is a core for $H$ and the
domain conditions follow from the formula
$e^{-\i\alpha A}He^{\i\alpha A}=e^{\theta\alpha}H$.  The commutation
relation is $[H,\i A]=\theta H$, which yields the estimate of Example
1.

\subsection{Example 5:\textnormal{ Fractional Laplacian in $\R^n$.}}$ $

For $0\,{<}\,s\,{<}\,2$, let $H=(-\Delta)^{s/2}$ with domain the
Sobolev space $\ch^s(\R^n)$ and consider
$A=-\tfrac{\i}{2}(x\cdot\nabla+\nabla\cdot x)$.  Then
$\mathscr{S}=C^\infty_\rmc(\mathbb{R_+})$ is a core for $H$ and
homogeneity of $H$ with respect to $A$ implies that $[H,\i A]=sH$. The
estimate of Example 1 follows.

\subsection{Example 6:\textnormal{ Multiplication by $\lambda$ in
    $L^2(\R_+,\d\mu)$.}} $ $

Let $H=\lambda$ and $A=-\frac{\i}{2}(\lambda\partial_\lambda+g(\lambda))$ on
$L^2(\R_+,\d\mu)$, where $g$ is to be determined. Assume that
$\d\mu=h(\lambda)\d\lambda$, for some non-vanishing function $h$ of
class $C^1(\R_+)$. It can be shown that if $g$ satisfies the
relation $g(\lambda)=\lambda\frac{h'}{h}+1$, then $A$ is
self-adjoint in $L^2(\R_+,\d\mu)$. For instance, if
$h(\lambda)=\lambda^N$ then choose $g(\lambda)=N+1$.  If $g$ is a
bounded function, $\mathscr{S}=C^\infty_\rmc(\mathbb{R_+})$ is a
core for $A$ and the commutation relation is $[H,\i A]=2H$. For
$z\in\rho(H)$ the function $(\lambda-z)^{-1}$ is smooth and has
bounded derivative on $\R_+$, hence the domain invariance $R(z)
D(A)\subset D(A)$ can be easily checked. Therefore $H$ is of class
$C^1(A)$, which gives the estimate of Example 1.

\subsection{Example 7:\textnormal{ Dirac operator in $L^2(\R^3;\C^4)$.}} $ $

We consider the Dirac operator for a spin-$1/2$ particle of mass $m>0$
given by $H=\alpha\cdot P+\beta m$ on $\ch=L^2(\R^3;\C^4)$, where
$\alpha=(\alpha_1,\alpha_2,\alpha_3)$ and $\beta$ denote the
$4\times 4$ Dirac matrices. The domain of $H$ is the Sobolev space
$\ch^1(\R^3;\C^4)$ and it is known that
$\sigma(H)=\sigma_{\rm{ac}}(H)=(-\infty,-m]\cup[m,\infty)$. See the
book of Thaller \cite{Tha}.

The Foldy-Wouthuysen transformation $U_{\rm FW}$ maps the free Dirac
operator into a $2\times 2$ block form. Consider the Newton-Wigner
position operator $Q_{\rm NW}$ defined as the inverse FW-tranformation
of multiplication by $x$, that is,
$Q_{\rm NW}=U^{-1}_{\rm FW}QU_{\rm FW}$. Using $A=Q_{\rm NW}$, then
$H$ is of class $C^1(A)$ and direct calculation shows that
$[H,\i A]=\sqrt{H^2-m^2}H^{-1}$ \cite{RT5}. The following decay
estimate follows from this commutation relation.

\begin{proposition}
  Let $H$ and $A$ as above. Then for $u\in D(A)\cap\ce$ such that
  $Au\in\ce$, one has the estimate
  $|\psi_u(t)|\leq C_u\jap{t}^{-1/2}$.
\end{proposition}

\proof Let $\varphi\in C_c^\infty([m,\infty))$ real and equal to one
on a small interval $[m,m+\epsilon]$ and set $\phi=\varphi(H)$,
$\phi^\perp=1-\phi^2$. For simplicity we assume $u$ in the subspace of
positive energies, then
$\psi_u=\psi_{\phi u}+\braket{\phi^\perp u}{\e^{\i tH}u}$. For the
high-energy region
\begin{align*}
t\braket{u}{\phi^\perp\e^{\i tH}u}&=\braket{\phi^\perp u}{t\e^{\i tH}u}\\
&=\braket{\phi^\perp u}{H(H^2-m^2)^{-1/2}[\e^{\i tH},A]u}\\
&=\braket{H(H^2-m^2)^{-1/2}\phi^\perp\e^{-\i tH} u}{Au}-\braket{AH(H^2-m^2)^{-1/2}\phi^\perp u}{\e^{\i tH}u},
\end{align*}
and it follows that
$|\braket{u}{\phi^\perp\e^{\i tH}u}|\leq C\jap{t}^{-1}$.

For energy close to $m$, assume that the support of $u$ in a spectral
representation of $H$ is contained in a compact interval.

Note that $[\e^{\i t(H-m)},A]=t\sqrt{H^2-m^2}H^{-1}\e^{\i t(H-m)}$ as
continuous forms on $D(H)$.

Define the auxiliary function $\psi(t)=\braket{u}{\e^{\i
    t(H-m)}u}$. Then
\begin{align*}
-\i t\psi'(t)&=\braket{u}{(H-m)t\e^{\i t(H-m)}u}\\
&=\braket{(H-m)^{1/2}u}{H(H+m)^{-1/2}[\e^{\i t(H-m)},A]u}\\
&=\braket{(H-m)^{1/2}H(H+m)^{-1/2}\e^{-\i t(H-m)}u}{Au}\\
&\quad\,-\braket{A(H-m)^{1/2}H(H+m)^{-1/2}u}{\e^{\i t(H-m)}u}.
\end{align*}
The right hand side is in $L^2_t$ because $H\in C^1(A)$ and $u$ is compactly supported so Lemma \ref{lm:scom} applies. We conclude that $|\psi(t)|\leq C\jap{t}^{-1/2}$ for all $t$ and since $|\psi_u|=|\psi|$ the result is proven. \qed

\subsection{Example 8:\textnormal{ Wave equation in $\R^n$.}} $ $

For $H>0$ consider the equation
$$
\text{(WE)}\left\{
\begin{array}{ll}
\partial_{tt}u+H^2u = 0\\
u(0)=f\\
\partial_tu(0)=g.\\
\end{array}
\right.
$$
Assume $\ch=L^2(\R^n)$. Define $u_1(t):=\cos(tH)$,
$u_2(t):=\frac{\sin(tH)}{H}$. Then
$u(t):=u_1(t)f+u_2(t)g$ is a solution to (WE).

For $f,\,g\in\mathcal{H}$ define the function
$\psi_{f,g}(t):=\braket{f}{u_1(t)f}+\braket{f}{u_2(t)g}$
and the subspace $\ce=\{u\in\ch \mid [u]_H<\infty\}$, where
$[h]_H=\|\braket{h}{u_1(t)h}\|_{L_t^2}^{1/2}+\|\braket{h}{u_2(t)h}\|_{L_t^2}^{1/2}$.

\begin{proposition} Let $H$ and $A$ be self-adjoint operators,
  assume $H\in C^1(A)$ and the commutation relation $[H,\i A]=cH$,
  with $c\neq 0$. Then for $f,g\in D(A)\cap\ce$ such that
  $Af,Ag\in\ce$, one has the estimate $|\psi_{f,g}(t)|\leq
  C_{f,g}\jap{t}^{-1/2}$.
\end{proposition}
\begin{proof}
  Similarly to Proposition \ref{pr:com}, the following two
  sesquilinear forms restricted to $D(A)\cap D(H)$ extend to
  continuous forms on $D(H)$ satisfying the identities
\begin{eqnarray*}
[\cos(tH),\i A] &=& -ctH\sin(tH)\\
\left[\dfrac{\sin(tH)}{H},\i A\right]&=&ct\cos(tH)-c\,\dfrac{\sin(tH)}{H}.
\end{eqnarray*}

We will use Corollary \ref{co:psidecay} for $f,g\in
D(|H|^{1/2})$. Clearly $\psi_{f,g}\in L^2(\mathbb{R})$. Now we
calculate
\begin{eqnarray*}
  ct\psi_{f,g}'(t)&=&-\braket{f}{ctH\sin(tH)
    f}+\braket{f}{ct\cos(tH)g}\\
  &=&\braket{f}{[u_1,\i A]f}+\braket{f}{[u_2,\i
    A]g}+c\langle f,u_2g\rangle\\
  &=&\braket{u_1f}{\i Af}+\braket{\i Af}{u_1f}+\braket{
    u_2f}{\i Ag}+\braket{\i Af}{u_2g}+c\braket{
    f}{u_2g}.
\end{eqnarray*}
Thus $c\|\delta\psi_{f,g}\|_{L^2}\leq
C[f]_H([g]_H+[Ag]_H+[Af]_H)$. For $f,g$ not necessarily in $D(H)$ we
can proceed analogously to Proposition \ref{pr:decay3} using
$u_\epsilon=R_\epsilon u$ and letting $\epsilon\to 0$.\qed
\end{proof}

\subsection{Example 9:\textnormal{ Klein-Gordon equation in $\R^n$.}} $ $

Now we draw our attention to (WE) in the case
$H=\sqrt{-\Delta+m^2}$, for $m\,{>}\,0$. The vector space is again
defined as $\ce=\{u\in\ch \mid [u]_H<\infty\}$, where
$[h]_H=\|\braket{h}{\e^{\i tH}h}\|_{L_t^2}^{1/2}$. Let $A$ be the
generator of dilations, then $H$ is of class $C^1(A)$ and it can be
formally shown that $[H,\i A]=H-m^2H^{-1}$.

Let $u_1,\, u_2$ be as in (WE), define $\psi_{f,g}^{\,1}(t)=\braket{
  f}{u_1(t)f}$ and $\psi_{f,g}^{\,2}(t)=\braket{f}{u_2(t)g}$. We are
interested in the decay rate of
$\psi_{f,g}:=\psi_{f,g}^{\,1}(t)+\psi_{f,g}^{\,2}(t)$.

\begin{proposition}\label{prop:decay_KG}
For $H$ and $A$ defined as above and $f,g\in D(A)\cap\ce$ such that
$Af,Ag\in\ce$, then $|\psi_{f,g}(t)|\leq {C_{f,g}\jap{t}^{-1/2}}$.
\end{proposition}

\begin{proof} Note that this result is a direct consequence of
  Proposition \ref{th:decay5}. Higher-order decay estimates follow
  from Proposition \ref{th:kzero}. Here we present a direct proof.


We define the auxiliary function $\psi(t):=\braket{f}{\e^{\i
    t(H-m)}g}$ and we prove that the conditions of Corollary
\ref{co:psidecay} are satisfied. It is clear that $\psi\in
L^2(\mathbb{R})$ since $f,g\in\ce$. Assume $g\in D(H)$ and we estimate
\begin{eqnarray*}
 -\i t\psi'(t) &=& \braket{f}{t(H-m)\e^{\i t(H-m)}g}\\
  &=& \braket{f}{[\e^{\i t(H-m)},A]H(H+m)^{-1}g}\\
  &=&\braket{\e^{-\i t(H-m)}f}{AH(H+m)^{-1}g}+\braket{Af}{\e^{\i t(H-m)}H(H+m)^{-1}g}.
\end{eqnarray*}
By Lemma \ref{lm:scom} we conclude that $\|\delta\psi\|_{L^2}\leq
C([f]_H[g]_H+[f]_H[Ag]_H+[Af]_H[g]_H)$. For general $g\in D(A)$,
replace it by $g_\epsilon=R_\epsilon g$ and let $\epsilon\to 0$.

We conclude that $|\psi(t)|\leq C_{f,g}\jap{t}^{-1/2}$. Notice that
$|\braket{f}{\e^{\i tH}g}|=|\psi(t)|$ satisfies the same bound.

Now we prove the desired estimate. Observe that
$\psi_{f,g}^{\,1}(t) =\tfrac{1}{2}\left(\braket{f}{\e^{\i
      tH}f}+\braket{f}{\e^{-\i tH}f}\right)$, therefore
$|\psi_{f,g}^{\,1}(t)|\leq C_{f}\jap{t}^{-1/2}$.

For the second term, we write
$\psi_{f,g}^{\,2}(t) = \frac{1}{2}\left(\braket{f}{H^{-1}\e^{\i
      tH}g}+\braket{ f}{H^{-1}\e^{-\i tH}g}\right)$
and redefine the auxiliary function
$\psi(t):=\braket{f}{H^{-1}\e^{\i t(H-m)}g}$, which is in $L^2(\R)$ by
the spectral theorem. Now
\begin{eqnarray*}
 -\i t\psi'(t) &=& \braket{f}{tH^{-1}(H-m)\e^{\i t(H-m)}g}\\
  &=& \braket{f}{[\e^{\i t(H-m)},A](H+m)^{-1}g}\\
  &=&\braket{\e^{-\i t(H-m)}f}{A(H+m)^{-1}g}+\braket{Af}{\e^{\i
      t(H-m)}(H+m)^{-1}g},
\end{eqnarray*}
which again yields the estimate $|\psi(t)|\leq C_{f,g}\jap{t}^{-1/2}$, concluding the proof.\qed
\end{proof}

\section{Appendix}\label{s:appendix}

We prove here an auxiliary estimate. We consider functions $g$ defined
on $\R_+=(0,\infty)$ and denote $\|g\|_p$ their $L^p$ norms. Let
$\delta$ the operator $(\delta g)=xg'(x)$ acting in the sense of
distributions and set $\tilde{g}(t)=\int_0^\infty \e^{\i tx}g(x)\d x$
for $t>0$ (improper integral).

\begin{lemma}\label{lm:sqdecay}
$|\tilde{g}(t)|\leq |t|^{-1/2} 2^{3/2}(p-1)^{-1/2p}
\|g\|_p^{1/2}\|\delta g\|_q^{1/2}$ if
$1<p<\infty$ and $\frac{1}{p}+\frac{1}{q}=1$.
\end{lemma}

\proof We may assume that $g\in L^p$ and $\delta g\in L^q$. For any
$s>0$ we have
\begin{equation}\label{eq:A1}
\left|\int_0^s \e^{\i tx} g(x)\d x \right| \leq s^{1/q} \|g\|_p .
\end{equation}
Since $g\in L^p$ with $p<\infty$, there is a sequence $a_n\to\infty$
such that $g(a_n)\to0$ (otherwise $|g(x)|\geq c>0$ on a neighbourhood
of infinity, so $|g|^p$ cannot be integrable). Since $p>1$, after
integrating over $(s,a_n)$ and then making $n\to\infty$, we also obtain
\begin{equation}\label{eq:A3}
|g(s)| \leq \int_s^\infty |g'(x)| \d x
\leq (p-1)^{-1/p} s^{1/p-1} \|\delta g\|_q
\end{equation}
by H\"older inequality.  Then
\begin{align*}
\int_s^\infty \e^{\i tx} g(x) \d x
&=\lim_{a\to\infty} \int_s^a
\left(\frac{d}{dx}\frac{1}{\i t} \e^{\i tx}\right) g(x) \d x \\
&= \lim_{a\to\infty}\left[
\frac{\e^{\i ta}g(a)-\e^{\i ts}g(s)}{\i t}
-\frac{1}{\i t} \int_s^a \e^{\i tx}g'(x) \d x \right]   .
\end{align*}
We take here $a=a_n$ and make $n\to\infty$ to get
\[
-\i t  \int_s^\infty \e^{\i tx} g(x) \d x =
    \e^{\i ts}g(s) + \int_s^\infty \e^{\i tx}g'(x) \d x
\]
and then by using \eqref{eq:A3} two times we obtain
\[
\left|\int_s^\infty \e^{\i tx} g(x) \d x \right| \leq
2 (p-1)^{-1/p} s^{-1/q} t^{-1} \|\delta g\|_q .
\]
Let $\varepsilon>0$ and $s=\varepsilon^q/ t$. Then \eqref{eq:A1} and
the last inequality give
\[
|\tilde{g}(t)| \leq \varepsilon t^{-1/q} \|g\|_p
+ 2 (p-1)^{-1/p} \varepsilon^{-1} t^{-1/p} \|\delta g\|_q .
\]
The infimum over $\varepsilon>0$ of an expression $\varepsilon
a+\varepsilon^{-1}b$ is $2\sqrt{ab}$. This finishes the proof. \qed

\begin{corollary}\label{co:psidecay}
  If $\psi\in L^2(\mathbb{R})$ and $t\psi'(t)\in L^2(\mathbb{R})$ then
  $|\psi(t)|\leq C_\psi|t|^{-1/2}$ for
  $t\in\mathbb{R}\setminus\{0\}$.
\end{corollary}

\proof We use Lemma \ref{lm:sqdecay} with $p=2$ and $g$ equal to the
Fourier transform of $\psi$.  \qed

\begin{corollary}\label{co:kdecay}
  If a function $\psi$ is such that $t^{\frac{k-1}{2}}\psi(t)$ and
  $t^{\frac{k+1}{2}}\psi'(t)$ belong to $L^2(\mathbb{R})$ for some
  $k\geq1$ then $|\psi(t)|\leq C_\psi|t|^{-k/2}$ for all
  $t\in\mathbb{R}\setminus\{0\}$.
\end{corollary}

\end{document}